\documentclass[letterpaper, 10 pt, conference]{ieeeconf}

\usepackage[noadjust]{cite}

\IEEEoverridecommandlockouts                              
\usepackage{amsmath, amssymb, graphicx}

\usepackage{amsthm}
\usepackage{pgfplots}
\usepackage{tikz}
\usepackage{bbm}
\usepackage{tikz-3dplot}
\usepackage[normalem]{ulem}
\usepackage{pgfplots}
\usetikzlibrary{arrows.meta}
\usetikzlibrary{patterns}

\usepackage{fancyhdr}
\usepackage{multirow, multicol}
\let\labelindent\relax
\usepackage{enumitem}
\usepackage{tabu}
\newtheorem{assumption}{Assumption}

\newtheorem{definition}{Definition}
\newtheorem{theorem}{Theorem}

\newtheorem{lemma}{Lemma}
\newtheorem{example}{Example}
\newtheorem{corollary}{Corollary}
\newtheorem{remark}{Remark}
\newtheorem{claim}{Claim}
\usepackage{comment}
\usepackage{algorithm}
\usepackage{hyperref}
\usepackage{algpseudocode}
\usepackage{algorithm}
\usepackage{algpseudocode}
\usepackage{subcaption}
\usepackage{graphicx}
\newcommand{\vv}{\vert\vert}
\newcommand{\diam}{\text{diam}}
\newcommand{\cov}{\text{Cov}}
\newcommand{\B}{\mathbb{B}}

\newcommand{\W}{\mathbb W}

\allowdisplaybreaks

\title{\LARGE \bf On the Sample Complexity of Set Membership Estimation for Linear Systems with Disturbances Bounded by Convex Sets}

\author{Haonan Xu and Yingying Li
\thanks{Haonan Xu and Yingying Li are with the Department of Industrial and Enterprise Systems Engineering and the Coordinated Science Laboratory at the University of Illinois Urbana-Champaign. E-mail:  {\tt\small \{haonan9,yl101\}@illinois.edu}.}
}
\begin{document}

\maketitle
\thispagestyle{empty}
\pagestyle{empty}

\begin{abstract} This paper revisits the  set membership identification for linear control systems and establishes its convergence rates under relaxed assumptions on (i)  the persistent excitation requirement and (ii)  the system disturbances. In particular, instead of assuming persistent excitation exactly, this paper adopts the block-martingale small-ball condition enabled by randomly perturbed control policies to establish the convergence rates of SME with high probability. Further, we relax the assumptions on the shape of the bounded disturbance set and the boundary-visiting condition. Our convergence rates hold for disturbances bounded by general convex sets, which bridges the gap between the previous convergence analysis  for general convex sets and the existing convergence rate analysis for $\ell_\infty$ balls.  Further,
we validate our convergence rates   by several numerical experiments.

This manuscript contains supplementary content in the Appendix.
\end{abstract}


\section{Introduction}
Recent years have witnessed a revived interest in system identification, especially its  non-asymptotic analysis through the lens of  statistical learning \cite{li2023non,simchowitz2018learning,simchowitz2020naive,abbasi2011regret,musavi2024identification}. For example, least square estimation (LSE), as a popular \textbf{point estimator}, has been studied extensively for its sample complexity of parameter estimation for linear dynamical systems $x_{t+1}=A^*x_t +B^* u_t+w_t$ with unknown $(A^*, B^*)$ \cite{li2023non,simchowitz2018learning,simchowitz2020naive,abbasi2011regret}.\footnote{Though LSE's convergence rates for linear regression have long been known,  its  rates for  control systems were developed more recently due to the complications from temporal  correlations in control systems  \cite{simchowitz2018learning,abbasi2011regret}.} This further inspires research on LSE-based learning-enabled control algorithms and their non-asymptotic analysis \cite{simchowitz2020naive,dean2018regret}. These analyses on LSE  reveal an important insight into the exploration-exploitation tradeoff for linear  systems: randomly perturbed control policies, i.e., $u_t=\pi_t(x_t)+\eta_t$ with i.i.d. non-degenerate noises $\eta_t$, can achieve optimal system identification convergence rate for arbitrary  policies $\pi_t(\cdot)$ \cite{li2023non,simchowitz2018learning} and can further guarantee optimal regret under properly designed policies $\pi_t(\cdot)$ \cite{simchowitz2020naive}.

Besides the \textbf{point estimator} LSE, set membership estimation (SME) is a widely used \textbf{set estimator}  in  control literature \cite{kosut1992set,bai1995membership,livstone1996asymptotic,lorenzen2019robust,fogel1982value}. SME considers bounded noises $w_t\in \W$ and leverages the set $\W$ to characterize the  uncertainty/membership set of $(A^*, B^*)$.  SME is popular among robust adaptive control as it can  effectively learn and reduce the  uncertainty sets for  robust controllers, e.g., robust adaptive model predictive control (RAMPC) \cite{lu2019robust,lorenzen2019robust}, robust adaptive control barrier functions \cite{lopez2020robust}, etc. It is observed empirically that SME's uncertainty sets tend to be much smaller than the theoretical confidence bounds of LSE \cite{li2024setmembership}.

Compared to LSE, SME and its non-asymptotic performance have received  less attention from the perspective of learning-based control, especially under randomly perturbed control policies. Most of the previous  SME literature has to assume a persistent excitation (PE) condition to establish  SME's convergence rates \cite{rao1990arma, bai1998convergence,akccay2004size,lu2019robust,fogel1982value}. Though the PE condition is crucial for sufficient exploration,   guaranteeing  PE  at every time step can be challenging when there are other control objectives \cite{gonzalez2014model}. For example,   RAMPC controllers have to be modified by constrained optimizations to guarantee PE  \cite{marafioti2014persistently, lu2019robust, lu2023robust}. Therefore, it is tempting to consider a much simpler framework to achieve sufficient exploration, i.e., randomly perturbed policies, and explore SME's convergence rates in this framework.

Recently, there are some efforts on understanding SME's convergence rates under randomly perturbed control inputs \cite{li2024setmembership,musavi2024identification}. This scenario turns out to be much  more challenging to analyze than the  PE case due to (i) the lack of closed-form expressions for the sizes/diameters of the membership sets and (ii) the correlation among $\{w_t\}_{t\geq 0}$ when conditioning on a stochastic PE condition established for the randomly perturbed policies. To address these challenges, 
\cite{li2024setmembership}  develops a  novel stopping-time-based analysis, which is generalized to  linearly parameterized nonlinear system  \cite{musavi2024identification}.

However, these papers \cite{li2024setmembership,musavi2024identification}  assume that the support set $\W$   is a perfect $\ell_\infty$ ball, 
which is much more restrictive than merely assuming convex $\W$ for  convergence analysis in  \cite{lu2019robust,lu2023robust}.  The shape of $\W$ is important because  SME may fail to converge under an outer-approximation of $\W$ \cite{nayeri1994interpretable,lu2019robust,li2024setmembership,bai1995membership}, thus, one cannot simply outer-approximate an arbitrary support set by an $\ell_\infty$ ball to apply the convergence rates in \cite{li2024setmembership,musavi2024identification}. For instance, consider a  multi-zone Heating, Ventilation, and Air Conditioning (HVAC)  system: the  heat disturbances in different zones can be  different yet remain correlated. For example, external heat disturbances, such as solar radiation, differ between zones with and without windows. Besides, the internal disturbances that depend on occupancy levels may have a near-constant summation across all zones when   the total population inside the building is constant over certain time \cite{zeng2020identification}. As a result, heat disturbances are  bounded by irregular polytopes, making an $\ell_\infty$-ball outer approximation potentially too loose to guarantee convergence. 

\noindent\textbf{Contributions.} Our major contribution  is establishing SME's  convergence rates under general convex and compact $\W$, which bridges the gap between the  convergence analysis in \cite{lu2019robust,lu2023robust} for  convex and compact sets and the  convergence rate analysis  in \cite{li2024setmembership,musavi2024identification} for $\ell_\infty$ balls. 

To establish the non-asymptotic bounds, we first adopt a standard assumption on the distribution of $w_t$ for SME's convergence guarantees in   \cite{lu2019robust,lu2023robust}. Then, we propose a relaxed assumption, which not only generalizes the applicability of our convergence rates but also enables improved  rates in certain cases. Compared with the convergence rates in \cite{li2024setmembership}, our bound depends on a constant $\xi$ that reflects the geometric property of $\W$ and the probability distribution of $w_t$. Finally, we use simulation examples to validate our theoretical guarantees. 

This work  lays foundation for future design and non-asymptotic analysis of SME-based learning-enabled control. 

\noindent\textbf{Related works. }
SME is widely used for parameter estimation (especially parameters' uncertainty set estimation) in system identification \cite{karimshoushtari2020design, deller1989set, fogel1979system,bertsekas1971control}, which is also the focus of this paper. In particular, this paper is related to the convergence and the convergence rate analysis of SME for stochastic linear regression and linear system identification problems, which enjoy extensive research under the persistent excitation (PE) assumption \cite{bai1995membership, bai1994least,akccay2004size,lu2019robust,bai1998convergence}. 

It is worth mentioning that SME  is also applicable  for deterministic and bounded $w_t$ \cite{kosut1992set,bai1995membership,livstone1996asymptotic}. This attracts  applications for SME when the noises do not have nice statistical properties \cite{lu2023robust,lorenzen2019robust}. Some asymptotic guarantees exist for the deterministic case, e.g., \cite{livstone1996asymptotic,fogel1982value}, but most non-asymptotic analysis requires stochastic noises.

Besides, there are numerous variations of SME for system identification.  For example, there are computationally efficient SME algorithms by optimal bounding ellipsoid (OBE) approximation \cite{huang1987application, deller1994unifying} and  zonotopic approximation \cite{casini2017linear, bravo2006bounded}. Further, SME has been applied to  nonlinear systems \cite{musavi2024identification,tang2024uncertainty,chen2020nonlinear}, time-varying systems \cite{watkins1995ellipsoid, rao1993tracking},  unknown noise bounds \cite{li2024setmembership, lauricella2020set}, switched systems \cite{ozay2011sparsification}. SME also enjoys wide applications to adaptive control designs \cite{lu2019robust, yeh2024online, lu2023adaptive}.
In addition to system identification, SME is widely adopted in state estimation for output feedback systems \cite{bertsekas1971control,zhang2020set,tang2020set} and filtering \cite{chen2020nonlinear}.

\noindent\textbf{Notations. }
Let $(M_1, M_2)$ denote matrix/vector concatenation. Let $\vv \cdot\vv_F$ denote the Frobenius norm of a matrix. Let $\vv \cdot\vv_p$ denote the $\ell_p$ vector norm or the $L_p$ matrix norm for $1\leq p\leq \infty$. We use  $\tilde O(\cdot)$ to hide logarithmic terms. The interior of a set $E$ is denoted by $\mathring{E}$, and the boundary of $E$ is denoted by $\partial E$. $A\succ B$ means that  matrix $A-B$ is  positive definite. For any $p\in \{1,2,\infty, F\},\ r>0$, and vector/matrix $x$, let $\B_p^r(x)$  denote the closed $p$-norm ball with radius $r$ centered at $x$. Let $\B_p$ denote the unit ball centered at $0$. Let $\mathbb{I}_n$ denote the $n\times n$ identity matrix.

\noindent\textbf{Mathematical preliminaries. }
In a probability space $(\Omega, \mathcal{A}, \mathbb{P})$, we say $\{\mathcal{F}_i\}_{i\in\mathbb{N}}$ is a \textit{filtration} if $\forall\, i\in\mathbb{N}$, $\mathcal{F}_i$ is a sub-$\sigma$-algebra of $\mathcal{A}$ and $\forall\, i\leq j$ one has $\mathcal{F}_i\subseteq\mathcal{F}_j$. We denote the $\sigma$-algebra generated by a collection of random variables as $\sigma\{\cdots\}$.
A stochastic process $\{X_i\}_{i\in\mathbb{N}}$ is said to be \textit{adapted} to the filtration $\{\mathcal{F}_i\}_{i\in\mathbb{N}}$ if $\forall\, i\in\mathbb{N}$, the random variable $X_i$ is an $\mathcal{F}_i$-measurable function. Let $\tau$ be a random variable taking values from $[0,+\infty)$. We say $\tau$ is a \textit{stopping time} if $\forall t\geq 0$, we have $\{\tau\leq t\}\in\mathcal{F}_t$.

\section{Problem Formulation}
This paper considers a linear control system:
\begin{align}\label{equ: LTI}
	x_{t+1}=A^*x_t+B^*u_t+w_t, \quad t\geq 0, 
\end{align}
where $x_t\in\mathbb{R}^{n_x}$, $u_t\in\mathbb{R}^{n_u}$,  and $w_t\in\mathbb{R}^{n_x}$ respectively denote the state, the control input, and the process noise at time $t\geq 0$. The system parameters $A^*, B^*$ in \eqref{equ: LTI} are unknown and to be estimated. For ease of notation, we denote $\theta^*=(A^*,  B^*
)\in\mathbb{R}^{n_x\times n_z}$, $z_t=(
	x_t^\top,  u_t^\top)^\top \in\mathbb{R}^{n_z}$, where $n_z = n_x + n_u$. Thus, the system \eqref{equ: LTI} can be written as
\begin{equation}\label{eq2}
	x_{t+1} = \theta^*z_t + w_t.
\end{equation}

This paper focuses on a specific estimation algorithm, set membership estimation (SME), to quantify the uncertainty of $A^*, B^*$. SME is mainly applicable when  $w_t$ is bounded with a known bound $\W$ \cite{kosut1992set,bai1995membership,livstone1996asymptotic,lorenzen2019robust,fogel1982value}. This paper focuses on a known stage-wise bound $w_t\in \W$ for simplicity, and leaves the analysis for unknown bounds and/or more general bounds, e.g. energy constrained bounds \cite{fogel1979system}, as future work. 

We review the SME algorithm in detail below. Consider a sequence of single-trajectory data, $\{x_t,u_t, x_{t+1}\}_{t=0}^{T-1}$, generated from \eqref{equ: LTI}, where the horizon $T$ can be unknown beforehand. Let $\Theta_T$ denote the remaining uncertainty set of $\theta^*$ after the $T$ stages of data are revealed. $\Theta_T$ generated by SME, also called a membership set, is defined below:
\begin{align}\label{equ: SME algo}
	\Theta_T= \bigcap_{t=0}^{T-1}\left\{\hat\theta\in\mathbb{R}^{n_x\times n_z}\ :\ x_{t+1}-\hat\theta z_{t}\in\mathbb{W}\right\}.
\end{align}
Basically, SME rules out any $\hat \theta$ that is inconsistent with the linear dynamics by the noise bound $\W$. Notice that $\theta^*\in \Theta_T$ as long as $w_t \in \W$ for all $t\leq T-1$.

SME usually enjoys  better empirical performances than the confidence region formulas \cite{abbasi2011regret,simchowitz2020naive}  of LSE (see  \cite{li2024setmembership}). Our goal is to further study the convergence rates of SME by measuring the size of uncertainty sets with their diameters as defined below.

\begin{definition}[Diameter of a  set of matrices]
	For a set of matrices $\Theta\in\mathbb{R}^{n_x\times n_z}$, we define its \textit{diameter} to be $\diam(\Theta):=\sup_{\theta_1,\theta_2\in\Theta}\vv\theta_1 - \theta_2\vv_F$. 
\end{definition}

\begin{remark}[Computation of SME]
  Notice that the SME algorithm \eqref{equ: SME algo} may be computationally demanding for large $T$ and complicated $\W$. There are many approximation methods to reduce SME's computation complexity, e.g.  ellipsoidal approximations \cite{huang1987application, deller1994unifying,fogel1982value}, zonotopic approximations \cite{casini2017linear, bravo2006bounded}, fixed-complexity approximations \cite{lu2019robust}, etc. This paper only considers the vanilla version of SME and leaves  the convergence rate analysis of computationally efficient approximations  for future.
\end{remark}

In the following, we introduce the assumptions for our convergence rate analysis, including two assumptions throughout the paper and one classical assumption in the SME literature, which will be relaxed in Section \ref{sec: ts}.

\noindent\textbf{Assumptions throughout this paper.} Here, we introduce two assumptions that are considered throughout this paper. 
Firstly, we assume the stochastic properties of $w_t$ and the convexity and compactness of $\W$. 
\begin{assumption}[Compactly and convexly supported i.i.d. noise]\label{ass:iid}
	The additive noise $\{w_t\}_{t\geq 0}$ are identically and independently sampled from a compact and convex set $\mathbb{W}$ with a non-empty interior (i.e. $\mathring{\W}\neq \emptyset$) such that $\mathbb{E}(w_t) =
	0$ and $\cov\left(w_t\right)=
	\Sigma_w \succ \mathbf{0}$. The set $\W$ is known.
\end{assumption}
Though SME does not need any stochastic properties of $w_t$ to generate valid uncertainty sets that contain $\theta^*$, the i.i.d. assumption is standard in the literature when discussing SME's convergence rates for linear regression \cite{akccay2004size,lu2019robust,bai1998convergence}, as well as LSE's convergence rate analysis \cite{simchowitz2018learning,simchowitz2020naive,abbasi2011regret}. Besides, the convex and compact $\W$ is also commonly assumed in the SME literature \cite{lu2019robust,akccay2004size,bai1998convergence,bertsekas1971control}. 

Our next assumption is based on the block-martingale small-ball  (BMSB) condition introduced in \cite{simchowitz2018learning}. BMSB can be viewed as a stochastic version of the persistent excitation (PE) condition that is commonly assumed in the system identification literature \cite{lu2019robust,bai1998convergence,akccay2004size,bai1995membership}. This is because, similar to PE, BMSB requires sufficient exploration in all directions in the state space, which is crucial for the learning of the accurate system parameters.
\begin{definition}[BMSB condition \cite{simchowitz2018learning}]
    For a filtration $\{\mathcal{F}_t\}_{t\geq 0}$ and an $\{\mathcal{F}_t\}_{t\geq 0}$-adapted stochastic process $\{Z_t\}_{t\geq 0}$ such that $Z_t\in\mathbb{R}^d$, we say $\{Z_t\}_{t\geq 0}$ satisfies the $(k,\Gamma_{sb},p)$-BMSB condition for a $k\in\mathbb{Z}^+$, $\Gamma_{sb}\succ 0$, and $p\in [0,1]$ if, for any fixed unit vector $\lambda\in\mathbb{R}^d$, one has
        $\frac{1}{k}\sum_{i=1}^k\mathbb{P}\left(|\lambda^\top Z_{t+i}|\geq\sqrt{\lambda^\top\Gamma_{sb}\lambda}\middle| \mathcal{F}_t\right){\geq} p$ almost surely for all $t\geq 1$.

\end{definition}

\begin{assumption} [Bounded $z_t$ \& the BMSB condition]\label{ass:bmsb}

    {$\exists\, b_z$ such that $\forall\, t\geq 0,\ \vv z_t\vv_2\leq b_z$ almost surely.  For the filtration $\{\mathcal{F}_t\}_{t=0}^{T-1}$, where $\mathcal{F}_t:=\sigma\{w_0,\cdots, w_{t-1}, z_0,\cdots, z_t\}$, the adapted process $\{z_t\}_{t\geq 0}$\ satisfies the $(1,\sigma_z^2 \mathbb{I}_{n_z}, p_z)$-BMSB condition for some $\sigma_z>0$ and $p_z>0$.}
\end{assumption}

It has been shown in \cite{li2023non} that the BMSB condition can be achieved by any stabilizing controllers adding an i.i.d.  exploration noise $\eta_t$ with a positive definite covariance, i.e. $u_t=\pi_t(x_t)+\eta_t$.
The boundedness condition on $z_t$ can be naturally satisfied if the closed-loop system is bounded-input-bounded-output stable   since our disturbances are bounded.

\noindent\textbf{A classical assumption on the tight bound $\W$.}
Here, we review a classical assumption on the tightness of $\W$: pointwise boundary-visiting noises \cite{lu2019robust}. This assumption is important for the convergence of SME. Later, we will analyze convergence rates based on this assumption in Section \ref{sec: ptws}, and then discuss how to relax this assumption in Section \ref{sec: ts}.

\begin{assumption}[Pointwise boundary-visiting noise \cite{lu2019robust}]\label{ass: tight bound}
		$\forall\, \epsilon >0$, $\exists\, q_w(\epsilon)>0$ such that $\forall\, t\geq 0$, $\forall\, w^0\in\partial\mathbb{W}$: $\mathbb{P}\left(||w^0 - w_t||_2 < \epsilon\right)\geq q_w(\epsilon)$.
\end{assumption}
It is worth mentioning that Assumption \ref{ass: tight bound} does not require $w_t$ to visit $\W$'s boundary with a positive probability. We only assume that $w_t$ can visit any small neighborhood of the boundary with a positive probability. Three examples  satisfying Assumption \ref{ass: tight bound} are discussed below.

\begin{example}[Weighted $\ell_\infty$ ball]\label{eg1_1}
	For $w_t$ following a uniform distribution\footnote{Though only uniform distribution is considered, other distributions such as truncated Gaussian can also apply.} on $\mathbb{W} = [-a_1,a_1]\times\cdots\times[-a_{n_x}, a_{n_x}]$ with positive constants $a_1,\cdots, a_{n_x}$, we have $q_w(\epsilon)=O(\epsilon^{n_x})$.
\end{example}
\begin{example}[Weighted $\ell_1$ ball]\label{eg2_1} For $w_t$ uniformly distributed on $\mathbb{W} = \left\{w\in\mathbb{R}^{n_x}\ :\ \sum_{i=1}^{n_x} \frac{1}{a_i}|w_i| \leq 1\right\}$ with positive constants $a_1,\cdots, a_{n_x}$, we have $q_w(\epsilon)=O(\epsilon^{n_x})$.
\end{example}
\begin{example}[$\ell_2$ ball]\label{eg3_1} For $w_t$ uniformly distributed on $\mathbb{W} = \B_2^r(0)$ with fixed $r>0$ we have $q_w(\epsilon)=O(\epsilon^{n_x})$.
\end{example}

\label{sec: formulation}


\section{Convergence Rate under  Assumption \ref{ass: tight bound}}\label{sec: ptws}

In this section, we propose a non-asymptotic estimation error bound for the SME under the $\epsilon$-ball boundary-visiting noise Assumption \ref{ass: tight bound} and discuss its implications. 

\begin{theorem}\label{thm1}
Under Assumptions \ref{ass:iid}, \ref{ass:bmsb} and \ref{ass: tight bound}, for any $T, m$ such that $1\leq m <T$, for any $\delta>0$, the membership set $\Theta_T$  defined in \eqref{equ: SME algo} satisfies:
    \begin{equation}\label{eq4}
    \begin{split}
        \mathbb{P}\left(\diam(\Theta_T)>\delta\right)\leq\underbrace{\frac{T}{m}\Tilde{O}(n_z^{5/2})a_2^{n_z}\exp(-a_3m)}_{\text{Term 1}} + \\
        \underbrace{\Tilde{O}\left((n_xn_z)^{5/2}\right)a_4^{n_xn_z}[1-q_w(\frac{a_1\delta}{4})]^{\lceil\frac{T}{m}\rceil - 1} }_{\text{Term 2}},
    \end{split}
    \end{equation}
    where $a_1 = \frac{1}{4}\sigma_z p_z$, $a_2 = \max\{1, \frac{64 b_z^2}{\sigma_z^2 p_z^2}\}$, $a_3 = \frac{1}{8}p_z^2$, $a_4 = \max\{ 1, \frac{4b_z}{a_1}\}$; $b_z,\sigma_z,p_z$ are defined in Assumption \ref{ass:bmsb}.
\end{theorem}

We refer the reader to Appendix \ref{app:thm1} for the proof of Theorem \ref{thm1}.

\noindent\textbf{On Inequality \eqref{eq4}.} Inequality \eqref{eq4} provides an upper bound on the probability of the ``large diameter event" i.e. $\diam(\Theta_T)>\delta$. Therefore, the smaller the upper bound is, the better the SME performs in terms of estimation accuracy. Notice that Inequality \eqref{eq4} involves two terms: $Term \ 1$ bounds the probability that PE does not hold, and $Term \ 2$ bounds the probability that the uncertainty set is still large even though PE holds (see proof details in Appendix \ref{app:thm1}).

\noindent\textbf{On  the choices of $m$.} 
Notice that $Term \ 1$ in \eqref{eq4} decays exponentially in $m$ yet increases with $T$, while $Term \ 2$ decays exponentially in $\frac{T}{m}$. To ensure a small upper bound in \eqref{eq4}, one can choose $m$  at a scale of $m\geq O(\log T)$ (hiding other constant factors for intuitions here). The optimal choice of $m$ minimizes the upper bound \eqref{eq4}.

Though choosing $m$ seems complicated and   requires the knowledge of $T$, it is worth emphasizing that the choice of $m$ only affects our theoretical bound but does not affect the empirical performance of SME. Therefore, a suboptimal choice of $m$ will only increase the gap between our theory and the actual empirical performance, but will not degrade the empirical performance of SME.

\noindent\textbf{On convergence rates.} Theorem \ref{thm1} can be converted to convergence rates of $\diam(\Theta_T)$. Since \eqref{eq4} depends on $q_w(\cdot)$, to provide explicit bounds as illustrating examples, we consider $q_w(\epsilon)=O(\epsilon^p)$ for  $p> 0$, which includes numerous common distributions, e.g. Examples \ref{eg1_1}-\ref{eg3_1}.

\begin{corollary}\label{cor: qw=O(epsilon)1}
    If $q_w(\epsilon) = O(\epsilon^{p})$ for $p > 0$, given $m\geq O(n_z + \log T - \log\epsilon)$, then with probability no less than $1-2\epsilon$, one has
    \begin{equation*}
        \diam(\Theta_T) \leq \tilde O \left(\left(\frac{n_x n_z}{T}\right)^{1/p}\right).
    \end{equation*}
\end{corollary}

\noindent
\textbf{Convergence rates for Examples 1-3.} Recall that we have shown $q_w(\epsilon)=O\left(\epsilon^{n_x}\right)$ in Examples \ref{eg1_1}-\ref{eg3_1}. By Corollary \ref{cor: qw=O(epsilon)1}, we have the convergence rate: $\diam(\Theta_T)\leq \tilde{O}\left(\left(\frac{n_xn_z}{T}\right)^{1/n_x}\right)$. Though this provides valid convergence rate bounds for general convex $\W$, these bounds are much worse than LSE's convergence rate $\frac{1}{\sqrt T}$ and do not explain the promising empirical performance of SME (see Section \ref{sec: numerical} and \cite{li2024setmembership}). 

Therefore, in the next section, we  seek to improve the theoretical convergence rate bounds of SME.

\label{sec: thm1}


\section{Convergence Rate with Relaxed Assumption} \label{sec: ts}
This section provides a relaxed version of Assumption \ref{ass: tight bound}, which enables tighter convergence rates in some cases. In the following, Section \ref{subsec: relax ass}   introduces this relaxed version as Assumption \ref{ass: tbtl}. Then, Section \ref{subsec: conv rate}  discusses the corresponding estimation error bound in Theorem \ref{thm2}, followed by discussions on its differences from Theorem \ref{thm1}.

\subsection{A Relaxed Assumption on the Tightness of $\W$}\label{subsec: relax ass}
This subsection will introduce a relaxed version of Assumption \ref{ass: tight bound} in Section \ref{sec: formulation}. Our assumption relies on supporting half-spaces (SHS) and $\epsilon$-slices induced by SHS for convex sets. We first review these concepts below.
\begin{definition}[\textbf{Supporting half-spaces \& $\mathbf{\epsilon}$-slices}]\label{def: shs} Consider a convex and compact set $\mathbb{W}\subseteq\mathbb{R}^n$ with a non-empty interior (i.e. $\mathring{\mathbb{W}}\neq\emptyset$). For a boundary point $c\in\partial \mathbb{W}$ and a unit vector $h$ (i.e. $\vv h\vv_2 = 1$), we say the half-space
\begin{equation*}
    H(c,h) := \left\{x\in\mathbb{R}^n\ :\ h^\top x\geq h^\top c\right\}
\end{equation*}
is a \textbf{supporting half-space (SHS)} of $\mathbb{W}$ at point $c$ with normal vector $h$, if $\mathbb{W}\subseteq H(c,h)$. 

\begin{figure}[h]
     \centering
     \begin{subfigure}[B]{0.265 \textwidth}
         \centering
         \includegraphics[width=\textwidth]{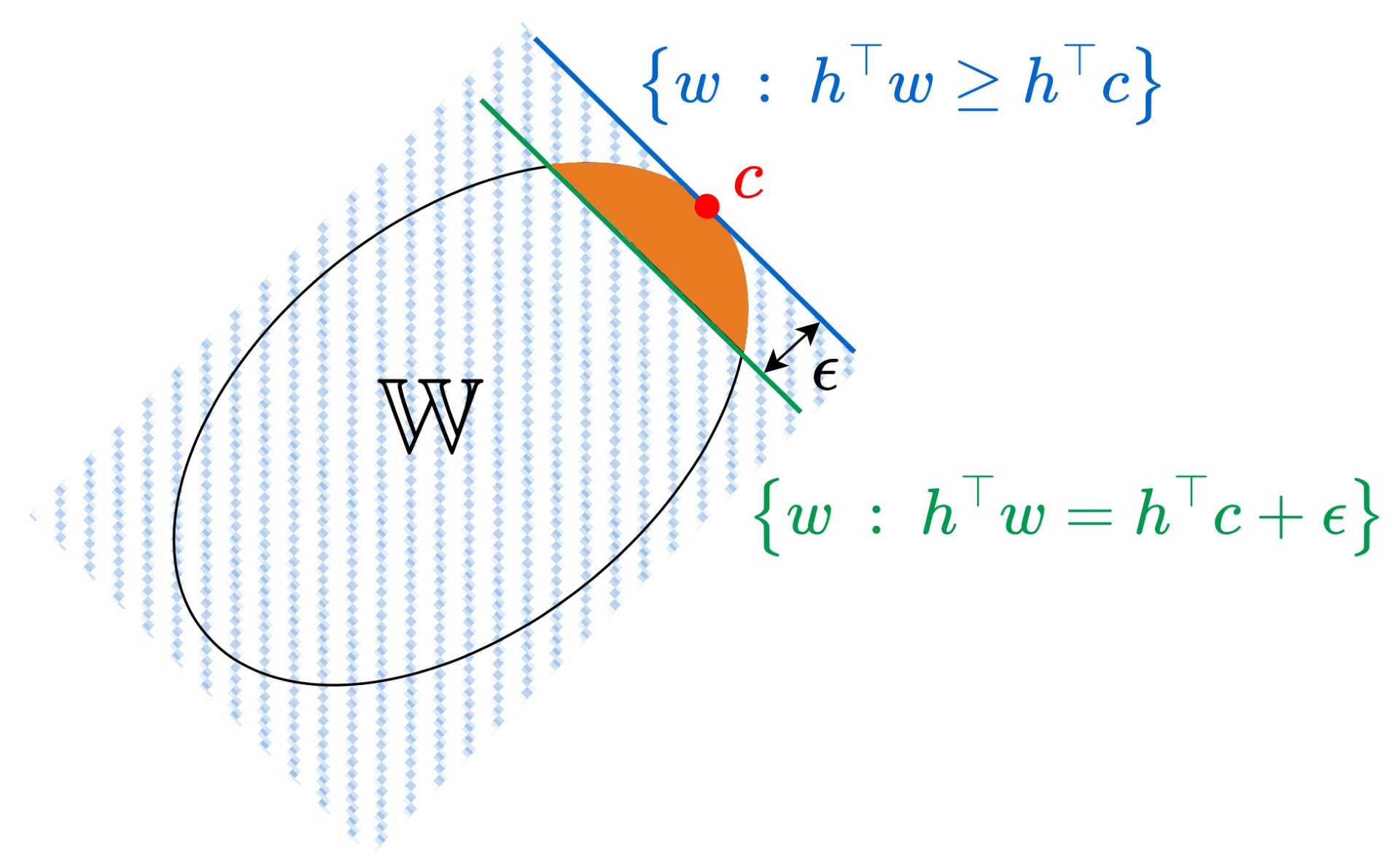}
         \caption{The $\epsilon$-slice induced by $H(c,h)$.}
         \label{fig: ts}
     \end{subfigure}
     \hfill
     \begin{subfigure}[B]{0.21 \textwidth}
         \centering
         \includegraphics[width=\textwidth]{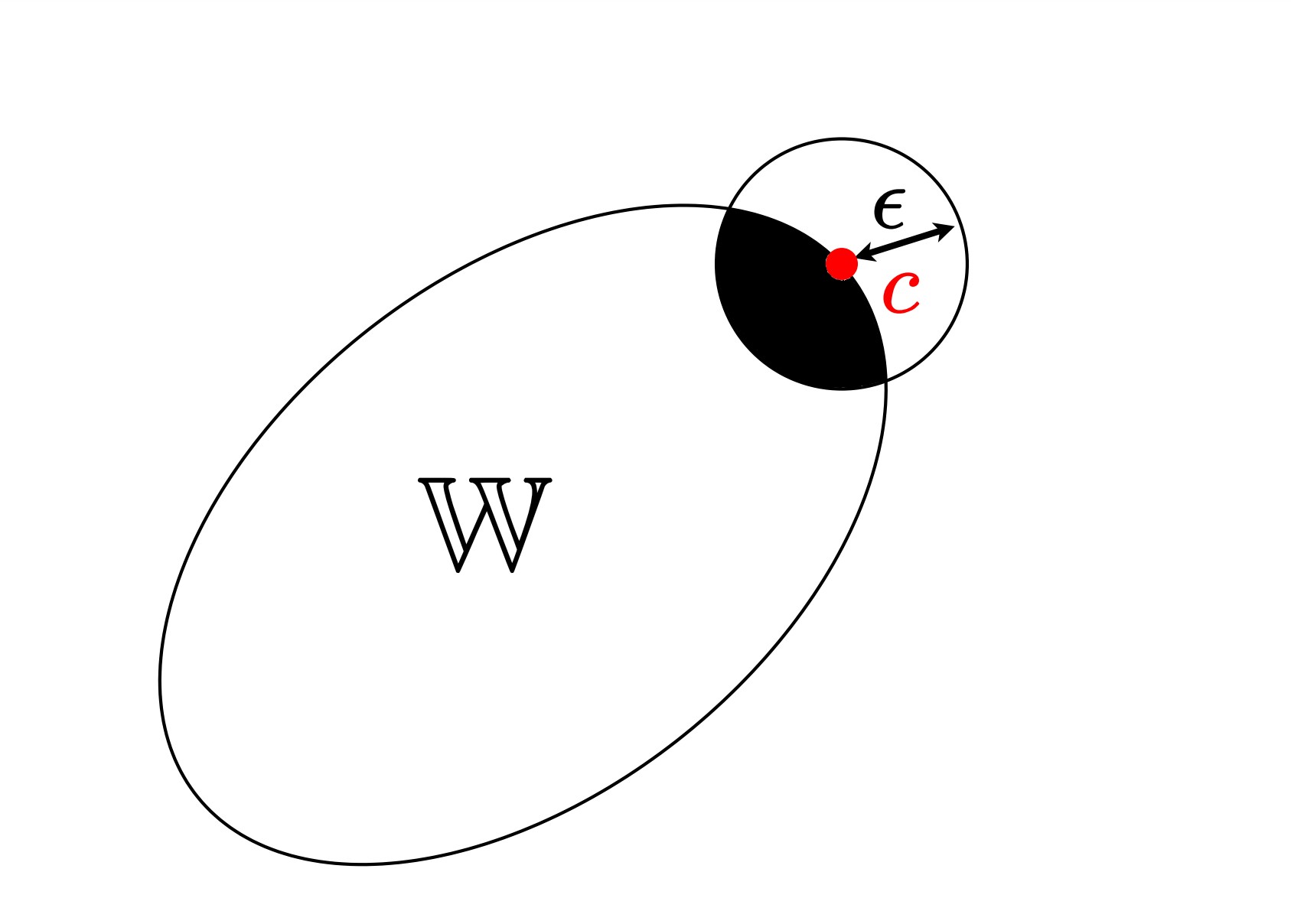}
         \caption{The $\epsilon$-ball at $c$.}
         \label{fig: ptws}
     \end{subfigure}
     \caption{(a) and (b) respectively demonstrate the $\epsilon$-slice and the $\epsilon$-neighborhood of a boundary point $c$. In Figure \ref{fig: compare assumptions}(a), the blue shaded half-space represents the SHS $H(c,h)$, and the orange-filled area is the $\epsilon$-slice induced by $H(c,h)$. In Figure \ref{fig: compare assumptions}(b), the black-filled area is the $\epsilon$-ball in Assumption \ref{ass: tight bound} at $c$. It can be shown that with the same $\epsilon$, the $\epsilon$-slice on the left is larger than the $\epsilon$-ball on the right (see Appendix \ref{app: epsilon}).
     }
     \label{fig: compare assumptions}
\end{figure}

Furthermore, $\forall\, \epsilon>0$, we define the $\mathbf{\epsilon}$\textbf{-slice} of $\mathbb{W}$ induced by $H(c,h)$ below:
\begin{equation*}
    S^\epsilon_{\mathbb{W}}(c,h) := \left\{x\in\mathbb{R}^n\ :\ h^\top c + \epsilon \geq h^\top x \geq h^\top c\right\}\cap\mathbb{W}.
\end{equation*}
\end{definition}

With Definition \ref{def: shs}, we can present a relaxed version of Assumption \ref{ass: tight bound} as in Assumption \ref{ass: tbtl} below.
\begin{assumption}[$\epsilon$-slice partial-boundary-visiting noise]\label{ass: tbtl}
  There exists a subset of boundary points $\mathcal{C}=\left\{c_1,\cdots,c_{\chi}\right\}\subseteq\partial\mathbb{W}$\footnote{Though we consider a finite number of boundary points in Assumption \ref{ass: tbtl}, the following Theorem \ref{thm2} remains true when $\mathcal{C}$ is infinite. Besides, Theorem \ref{thm1} may serve as a special case of when $\mathcal{C}$ is infinite.}, and a set of unit vectors $\mathcal{H} = \left\{h_1,\cdots,h_{\chi}\right\}$ satisfying the following properties. 
    \begin{itemize}
        \item [(i)]$\forall\, i\in\{1,\cdots,\chi\}$, $H(c_i, h_i)$ is a SHS of $\mathbb{W}$.
        \item [(ii)] $\displaystyle\bigcap_{i=1}^{\chi} H(c_i,h_i)$ is a compact set.
        \item [(iii)] $\forall\,\epsilon >0$, $\exists\, p_w(\epsilon) > 0$ such that $\forall\, t\geq 0$, $\forall\, i\in\{1,\cdots, \chi\}$,
    \begin{equation*}
        \mathbb{P}\left(w_t\in S^\epsilon_{\mathbb{W}}(c_i,h_i)\right)\geq p_w(\epsilon) > 0.
    \end{equation*}
    \end{itemize} 
\end{assumption}
Assumption \ref{ass: tbtl} is a relaxation of Assumption \ref{ass: tight bound} in two perspectives:
\begin{itemize}
    \item[(i)] While Assumption \ref{ass: tight bound} requires a non-vanishing probability in the neighborhood of every  point on $\partial \W$,  Assumption \ref{ass: tbtl} only requires it for a subset of $\partial \W$.
    \item[(ii)] The $\epsilon$-ball considered in Assumption \ref{ass: tight bound} is a subset of the $\epsilon$-slice considered in Assumption \ref{ass: tbtl} for the same $\epsilon$ (see Figure \ref{fig: compare assumptions} as an example; see Appendix \ref{app: epsilon} for detailed proof.). Therefore, $p_w(\epsilon)\geq q_w(\epsilon)$, which usually enables tighter convergence rates as shown later.
\end{itemize} 
To provide more intuitions for Assumption \ref{ass: tbtl}, we discuss the three examples in Section \ref{sec: formulation} below. 

\begin{example}[Weighted $\ell_\infty$ ball and weighted $\ell_1$ ball]\label{eg1}
In both Examples \ref{eg1_1} and  \ref{eg2_1}, we can choose $\mathcal C$ and $\mathcal H$ as follows:  let $\mathcal{C}$ consist of one non-extreme point per  facet of $\W$, and let $\mathcal{H}$ consist of the unit normal vectors of each facet of $\W$. It is shown in Appendix \ref{box} that  $p_w(\epsilon)=O(\epsilon)$ for both Examples \ref{eg1_1} and \ref{eg2_1}, which is much larger than $q_w(\epsilon)=O(\epsilon^{n_x})$ derived in  these examples. This enables better and tighter convergence rate bounds as shown later in this section.
\end{example}


\begin{example} [$\ell_2$ ball]\label{eg3}
Consider $w_t$  in Example \ref{eg3_1}. Let $\mathcal C=\partial \W$ and $\mathcal H$ be the corresponding normal vectors of the supporting half-spaces for the boundary points. It is shown in Appendix \ref{2norm} that $p_w(\epsilon)=O\left(\epsilon^{\frac{n_x+1}{2}}\right)$, which is much smaller than $q_w(\epsilon)=O\left(\epsilon^{n_x}\right)$ in Example \ref{eg3_1}, enabling a better theoretical convergence rate bound as shown later.
\end{example}

\subsection{Convergence Rate Analysis under Assumption \ref{ass: tbtl}}\label{subsec: conv rate}
From now on, we will assume that Assumptions \ref{ass:iid}, \ref{ass:bmsb}, and \ref{ass: tbtl} hold. We provide a new version of the estimation error bound in Theorem \ref{thm2} based on our relaxed Assumption \ref{ass: tbtl}. 

\begin{theorem}\label{thm2}Under Assumptions \ref{ass:iid}, \ref{ass:bmsb} and \ref{ass: tbtl}, $\forall\, T>m>0, \delta>0$, the membership set $\Theta_T$  defined in \eqref{equ: SME algo} satisfies:
    \begin{equation}\label{eqtl}
    \begin{split}
        \mathbb{P}\left(\diam(\Theta_T)>\delta\right)\leq\underbrace{\frac{T}{m}\Tilde{O}(n_z^{5/2})a_2^{n_z}\exp(-a_3m)}_{\text{Term 1}} + \\
        \underbrace{\Tilde{O}\left((n_xn_z)^{5/2}\right)a_5^{n_xn_z}[1-p_w(\frac{a_1\delta\xi}{4})]^{\left\lceil \frac{T}{m}\right\rceil - 1}}_{\text{Term 3}},
    \end{split}
    \end{equation}
    where $a_5 = \max\{ 1, \frac{4b_z}{a_1\xi}\}$, the projection constant $\xi = \min_{\vv x\vv_2 = 1}\max_{h\in\mathcal{H}} h^\top x$\footnote{In the case where $|\mathcal{H}| = +\infty$, the projection constant is similarly defined by $\xi = \min_{\vv x\vv_2 = 1}\sup_{h\in\mathcal{H}} h^\top x$.}, and $a_1,a_2,a_3$ are defined in Theorem \ref{thm1}.
\end{theorem}
We refer the reader to Appendix \ref{app: thm2} for the proof of Theorem \ref{thm2}.

\noindent\textbf{Differences between Theorems \ref{thm1} and  \ref{thm2}.} Both \eqref{eq4} and \eqref{eqtl} consist of two terms, where $Term \ 1$ is the same, but the second term is different. $Term\ 3$ differs from $Term\ 2$ in two aspects: (i) $Term \ 3$ depends on $p_w(\cdot)$ while $Term\ 2$ depends on $q_w(\cdot)$, and (ii) there is an additional projection constant $\xi$ in $p_w(\cdot)$ and $a_5$ of $Term \ 3$. Notice that both $p_w$ and $\xi$ depend on our choices of $\mathcal C$ and $\mathcal H$, which will be discussed in more details below.

\noindent\textbf{On the projection constant.} $\xi$ depends on our choices of $\mathcal C$ and $\mathcal H$. By definition, we have  $\xi := \displaystyle\min_{\vv x\vv_2 = 1}\max_{h\in\mathcal{H}}h^\top x$. It can be shown that $0<\xi \leq 1$ (See Appendix \ref{app: xi}). Notice that, by adding more boundary points to $\mathcal C$ and more directions to $\mathcal H$, we can increase $\xi$. As an extreme case, if we choose all the points on the boundary and choose all the unit vectors as $\mathcal H$, then $\xi=1$. However, if we choose too many points in $\mathcal C$,  $p_w(\cdot)$ may decrease since it is a uniform lower bound for all points in $\mathcal C$. This suggests   a tradeoff on the choices of $\mathcal C$ and $\mathcal H$.

\noindent\textbf{On the choices of $\mathcal C$ and $\mathcal{H}$.} There is a trade-off on choosing $\mathcal C$ and $\mathcal{H}$ to minimize the upper bound in \eqref{eqtl}. On the one hand, as discussed earlier, since \eqref{eqtl} decreases with $\xi$, it is tempting to add more boundary points to $\mathcal{C}$ to increase $\xi $ to reduce the upper bound in \eqref{eqtl}. On the other hand, notice that \eqref{eqtl} decreases with $p_w(\cdot)$, but having  more points in $\mathcal C$ might decrease the boundary-visiting probability density $p_w(\cdot)$, which induces a larger upper bound \eqref{eqtl}. For instance, consider Example \ref{eg1_1}, if we add a SHS at a degenerate point that is not parallel to any facets of the weighted $\ell_\infty$ ball, $p_w(\epsilon)$ will decrease to $O(\epsilon^{n_x})$. Therefore, the best estimation error bound induced by Theorem \ref{thm2} is by optimizing over the choices of $\mathcal C$ and $\mathcal H$.

Though it is complicated to optimize the choices of $\mathcal C$ and $\mathcal H$ in general, it is worth emphasizing that these choices only affect our theoretical bounds and do not affect the empirical performance of SME. 

Besides, for some special cases, we have some intuitions on optimizing $\mathcal C$ and $\mathcal H$. For example, for any polytopic $\W$,  we choose exactly one non-extreme point as $c$ for every facet of $\mathbb{W}$, then the unit normal vector $h$ at $c$ is unique, and $p_w(\epsilon)=O(\epsilon)$ is the largest possible value.

\noindent\textbf{On the convergence rates.} Similar to Corollary \ref{cor: qw=O(epsilon)1}, we discuss the explicit convergence rates of SME induced by Theorem \ref{thm2} in the following corollary.

\begin{corollary}\label{cor: qw=O(epsilon)}
    If $p_w(\epsilon) = O(\epsilon^p)$ for  $p> 0$, given $m\geq O(n_z + \log T - \log\epsilon)$,  with probability at least $1-2\epsilon$, one has
    \begin{equation*}
        \diam(\Theta_T) \leq \tilde O \left(\frac{1}{\xi}\left(\frac{n_x n_z}{T}\right)^{\frac{1}{p}}\right).
    \end{equation*}
\end{corollary}

\noindent\textbf{Convergence rates for Examples \ref{eg1}-\ref{eg3}.}
For Example \ref{eg1}, we have $p_w(\epsilon) = O(\epsilon)$. By Corollary \ref{cor: qw=O(epsilon)}, we have the convergence rates of both the $\ell_1$ and the $\ell_\infty$  ball supports are $\tilde{O}\left(\frac{n_xn_z}{\xi T}\right)$. Specially, for the $\ell_\infty$ ball, $\xi = \frac{1}{\sqrt{n_x}}$, which is consistent with the result in \cite{li2024setmembership}. As for Example \ref{eg3}, we have $p_w(\epsilon) = O\left(\epsilon^{\frac{n_x+1}{2}}\right)$. Hence, for the $\ell_2$ ball, we have $\xi = 1$, which leads to a convergence rate of $\tilde{O}\left(\left(\frac{n_xn_z}{T}\right)^{\frac{2}{n_x+1}}\right)$.

Though only uniform distribution is considered in the above two examples, some other noise types also satisfy the $O(\epsilon)$ boundary assumption (e.g. truncated Gaussian). By Corollary \ref{cor: qw=O(epsilon)}, the convergence rates of both the weighted $\ell_\infty$-ball and the weighted $\ell_1$ ball are of $\tilde{O}\left(\frac{n_xn_z}{\xi T}\right)$. In fact, a uniform / truncated Gaussian noise supported on a polytope always results in a $\tilde{O}\left(\frac{n_xn_z}{\xi T}\right)$ convergence rate (simply consider $\mathcal{C}$ to be the set of any non-extreme point on every facet of $\W$). Proofs of Corollary \ref{cor: qw=O(epsilon)} and all the examples can be found in Appendix \ref{app: cor}.

\begin{remark}
We first introduce Assumption \ref{ass: tight bound} then  Assumption \ref{ass: tbtl} because (i) we demonstrate the benefits of  Assumption \ref{ass: tbtl} by comparison with the classical one, (ii) Assumption \ref{ass: tight bound} and  Theorem \ref{thm1}  are  simpler, e.g. not involving SHS or $\xi$.  
\end{remark}

\section{Numerical Experiments}\label{sec: numerical}

This section provides numerical examples to test our theoretical bounds of SME in comparison with LSE's confidence regions as benchmarks. We consider two simulation settings: a randomly generated example (Section \ref{subsec: different xis}) and a linearized Boeing 747 model based on \cite{ishihara1992design} (Section \ref{subsec: boeing}).

\noindent\textbf{General setup for all experiments.} Throughout this section, we consider LSE confidence region provided in Theorem 1 of \cite{abbasi2011regret} as our benchmark. In particular, we adopt  a 95\% confidence level and a regularization parameter $\lambda = 10^{-3}$. Further, we use the trace of $w_t$'s covariance  matrix as a lower bound for the variance proxy of  $w_t$ \cite{arbel2020strict}. Therefore, the LSE confidence regions we plot here are smaller than those using the exact variance proxies.  For each experiment setting, we independently and identically generate $5$ sample trajectories of length $T = 2000$ with $x_0 = 0$, and $u_t$ identically and independently generated from the same distribution as $w_t$. The curves and the shaded areas in the plots respectively represent the empirical means and standard deviations.

\subsection{A Randomly Generated Example}\label{subsec: different xis} 

In this  experiment, we generate  $(A^*, B^*)$ with $n_x=n_u=2$ randomly. The generated parameters are
\begin{equation*}
    A^* = \begin{bmatrix}
        0.377 & -0.788\\
        -0.533 & 0.143
    \end{bmatrix},\ B^* = \begin{bmatrix}
        1.067 & -0.366\\
        0.520 & -0.480
    \end{bmatrix}.
\end{equation*}
All the trajectories adopt the same distribution of $w_t$: uniformly distributed on $\W=\mathbb{B}_2$, i.e. the unit $\ell_2$ ball. However, when implementing SME according to \eqref{equ: SME algo}, we use (i) the exact $\W$, and two tight outer-approximations of $\W$: (ii) a regular quadrilateral ($4$ linear constraints) circumscribing $\mathbb{B}_2$, and (iii) a regular hexadecagon ($16$ linear constraints) circumscribing $\mathbb{B}_2$. It can be shown that these three cases enjoy the same $p_w(\cdot)$ but different $\xi$. Further, we have $0 < \xi_4 < \xi_{16} <  1$, where $\xi_i$ denotes the projection constant $\xi$ for  SME with $i$-constrained $\W$. According to Example \ref{eg3}, $\xi=1$ for $\B_2$.

 Figure \ref{fig: different xis} compares the diameters of the uncertainty sets calculated by SME with the exact $\W$ or with the two outer-approximations of $\W$, and the 95\% confidence regions of LSE. It can be observed that SME  generates smaller uncertainty sets than LSE's confidence regions even when using tight outer-approximations, demonstrating the effectiveness of SME in reducing the parameter uncertainty. 

 \noindent\textbf{On the projection constant $\xi$.} One can observe from Figure \ref{fig: different xis} that the membership set trained with the exact noise bound (i.e. $\W = \mathbb{B}_2$) has the least diameter, and that the one trained with a hexadecagon (16 constraints) performs better than the one with a quadrilateral (4  constraints). Since $1>\xi_{16}>\xi_4$, Figure \ref{fig: different xis} shows that SME performs better with a larger $\xi$. This is consistent with our   Theorem \ref{thm2} and Corollary \ref{cor: qw=O(epsilon)}, which show a decreasing  diameter of SME  with respect to $\xi$.

\begin{figure}
    \centering
    \includegraphics[width=0.8\linewidth]{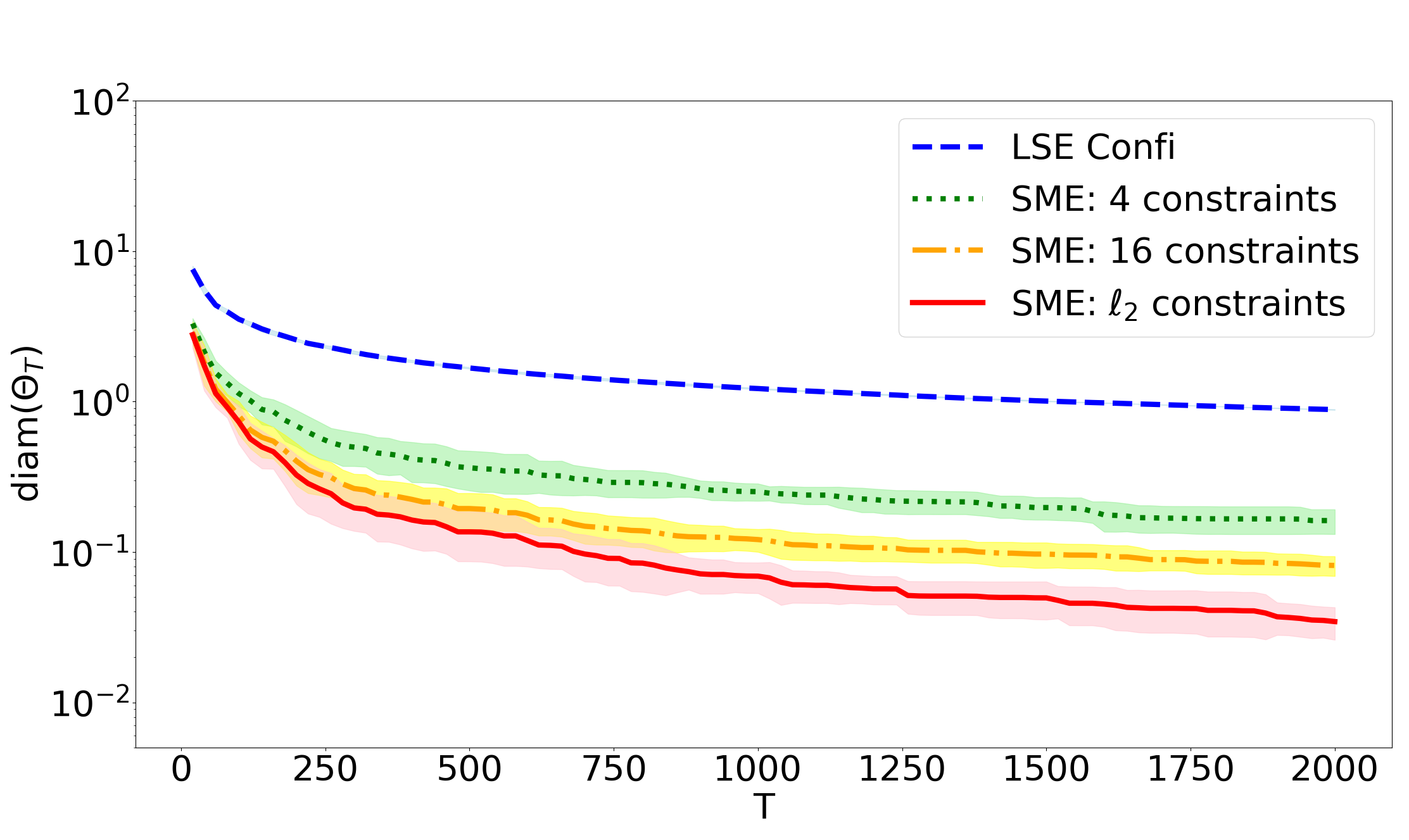}
    \caption{Comparion of LSE's  confidence regions  with SME with different $\W$ that enjoys same $p_w(\cdot)$ but different $\xi$.}
    \label{fig: different xis}
\end{figure}
\subsection{A Linearized Boeing 747 Flight Control System}\label{subsec: boeing}
In the following two numerical experiments, we consider identifying a linearized Boeing 747 flight control model with $n_x = 4$, $n_u=2$ based on \cite{ishihara1992design}. Besides, we always adopt the exact $\W$ when implementing the SME algorithm \eqref{equ: SME algo}.

The second numerical experiment assesses the performance of SME under noise bounds with different sizes (i.e. different $p_w(\cdot)$) but of the same shape (i.e. the same $\xi$). In particular, we consider  $w_t$ generated by two truncated standard Gaussian distributions, where all distributions are truncated from standard normal distribution but on different $\W$: $\ell_1$ balls with different radii $w_{\max}\in\{1,3\}$.

\noindent\textbf{On $p_w(\cdot)$.}
 Figure \ref{fig: different pws} compares SME   and LSE's confidence bounds when $w_t$ follows truncated standard Gaussian distributions with different sizes of $\W$. Figure \ref{fig: different pws} shows that SME generates smaller uncertainty sets with a smaller  truncation radius $w_{\max}$. Notice that $p_w(\cdot)$ will decrease as $w_{\max}$ increases in a truncated version of a standard Gaussian distribution. Therefore, the trend observed in Figure \ref{fig: different pws}   is  consistent with our 
Theorem \ref{thm2} and Corollary \ref{cor: qw=O(epsilon)}, which show that  the estimation errors of SME decrease with $p_w(\cdot)$. 
In addition, Figure \ref{fig: different pws} shows that LSE's confidence bounds also increase with $w_{\max}$. This is because  increasing $w_{\max}$ causes increasing variances of $w_t$, and thus causes increasing confidence bounds in \cite{abbasi2011regret}.

\begin{figure}
    \centering
\includegraphics[width=0.8\linewidth]{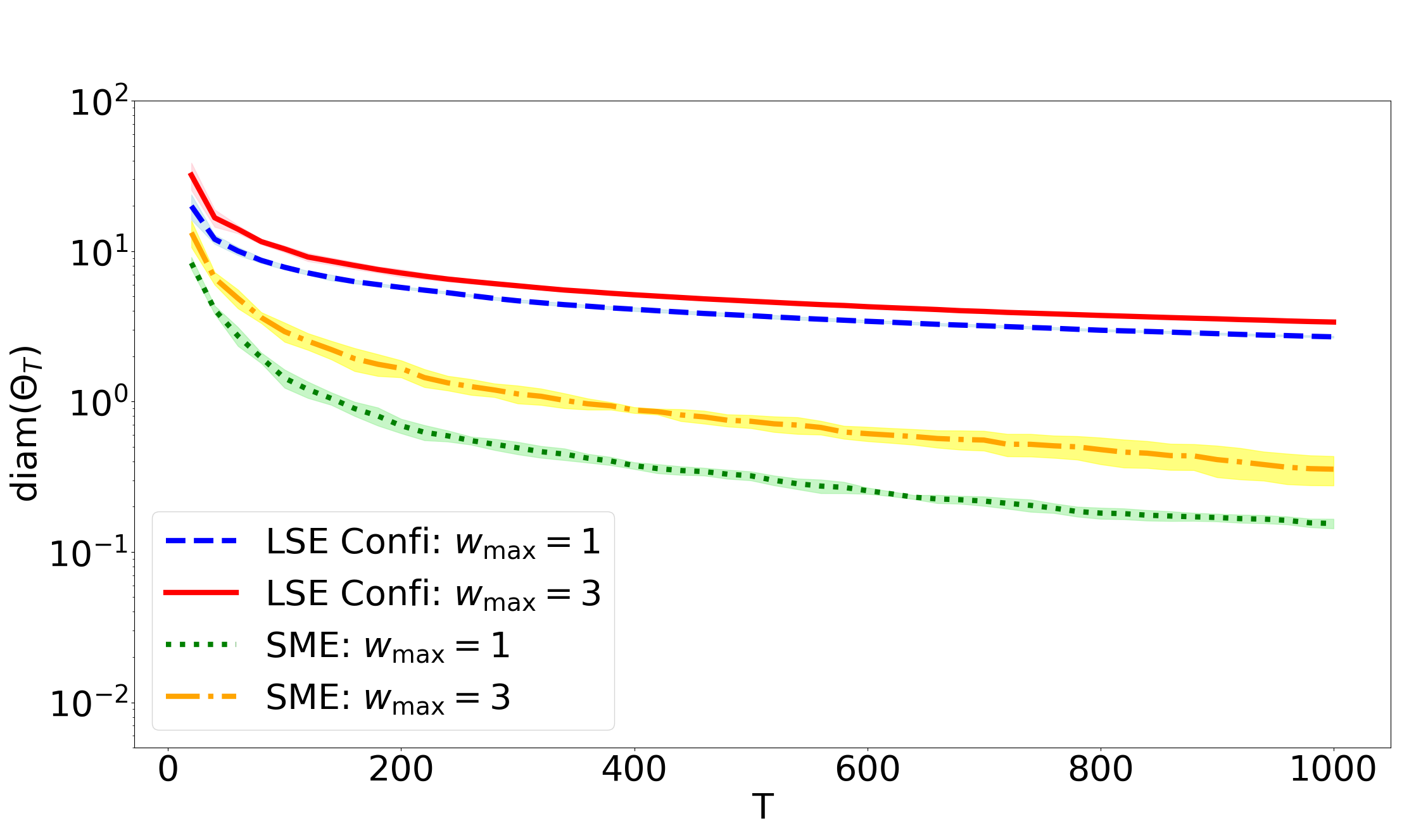}
    \caption{Comparison of SME and LSE's confidence regions for truncated standard Gaussian on $\ell_1$ balls with $w_{\max}=1,3$.}
    \label{fig: different pws}
\end{figure}

\noindent\textbf{Discussions on  distributions and  $\ell_2$ balls.}
 Figure \ref{fig: different supports} compares SME and LSE's confidence bounds under  four  distributions of $w_t$:   uniform distributions  on the unit $\ell_1$ ball $\mathbb{B}_1$ and the unit $\ell_2$ ball $\mathbb{B}_2$; and truncated standard Gaussian distributions on the unit $\ell_1$ ball $\mathbb{B}_1$ and the unit $\ell_2$ ball $\mathbb{B}_2$.  

 Comparing Figures \ref{fig:l1} and \ref{fig:l2}, SME outperforms LSE's confidence bounds in all cases. In addition, the trends are similar for both the uniform  and the truncated Gaussian case. 

 Finally, in  Figures \ref{fig:l1} and \ref{fig:l2}, we can observe that the differences in the SME's performances between $\ell_1$ and $\ell_2$ balls are not significant.
 This is interesting because our theoretical analysis can only guarantee a much worse bound for the $\ell_2$ case. Therefore, this numerical experiment indicates that, even though our theoretical bounds successfully  explains SME's good performances on polytopic $\W$, our convergence rates for the $\ell_2$ ball case are not tight and can be improved, which is an interesting future direction. 

\begin{figure}
     \centering
     \begin{subfigure}[t]{0.235\textwidth}
         \centering
         \includegraphics[width=\textwidth]{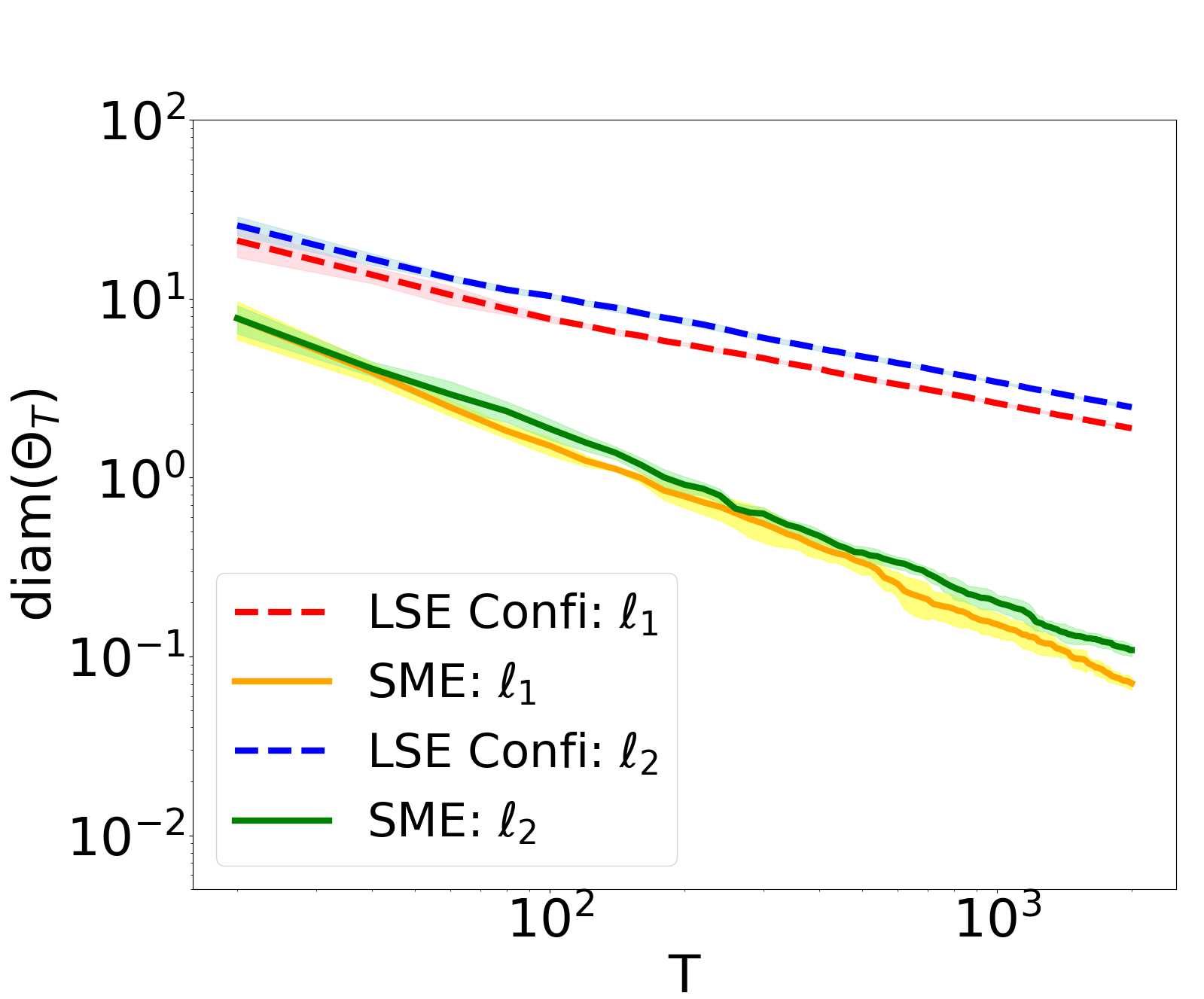}
         \caption{Uniform distributions.}
         \label{fig:l1}
     \end{subfigure}
     \hfill
     \begin{subfigure}[t]{0.235\textwidth}
         \centering
         \includegraphics[width=\textwidth]{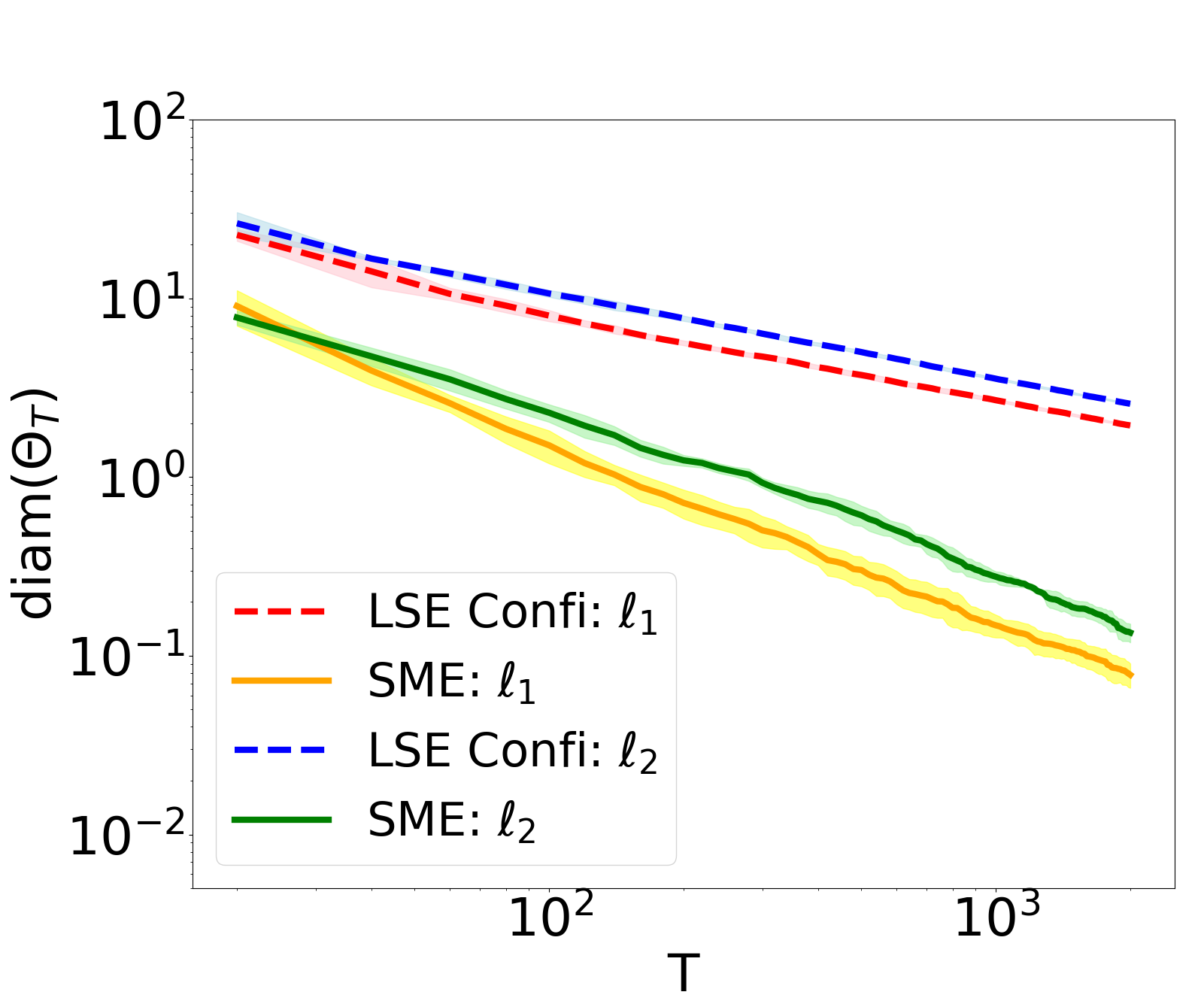}
         \caption{Truncated standard Gaussian.}
         \label{fig:l2}
     \end{subfigure}
     \caption{These two subfigures compare the LSE confidence bound with the SME under different noise distribution types (uniform and truncated standard Gaussian) and different noise supports ($\ell_1$ ball and $\ell_2$ ball).}
     \label{fig: different supports}
\end{figure}


\vspace{0.1cm}

\section{Conclusions and Future Directions}
This paper provides  non-asymptotic analysis for  set membership estimation for an unknown linear control system with  i.i.d. disturbances bounded by general convex and compact sets. We study both the classical assumption in \cite{lu2019robust} and propose a new relaxed assumption for better bounds. Future directions include: 1) improving the bound for $\ell_2$ balls; 2) non-asymptotic analysis of SME with non-stochastic disturbances; 3) proposing and analyzing more computationally efficient SME; 4) learning tight bounds $\W$ and its performance bounds;  5) fundamental limits; 6) regrets of robust adaptive controllers using SME, 7)  analyzing SME   under imperfect state observations, nonlinear systems, etc. 
\bibliographystyle{IEEEtran}
\bibliography{bibliography}


\appendix
\subsection*{List of Contents}
This appendix provides  proofs in the following order.
\begin{itemize}
    \item In Appendix \ref{app:thm1}, we provide a proof of Theorem \ref{thm1}.
    \item In Appendix \ref{app: thm2}, we provide a proof of Theorem \ref{thm2}.
    \item In Appendix \ref{app: xi}, we prove the strict positivity of $\xi$ as mentioned in the paragraph ``\textbf{on the projection constant}" after Theorem \ref{thm2}.
    \item In Appendix \ref{app: cor}, we prove Corollaries \ref{cor: qw=O(epsilon)1} and \ref{cor: qw=O(epsilon)}.
    \item In Appendices \ref{box} and \ref{2norm}, we respectively establish the asymptotic properties of $p_w(\cdot)$ and $q_w(\cdot)$ under $\ell_\infty$/$\ell_1$/$\ell_2$-bounded noise support settings.
    \item In Appendix \ref{app: epsilon}, we show the claim made in Figure \ref{fig: compare assumptions} that the $\epsilon$-ball is a subset of the $\epsilon$-slice with the same $\epsilon$.
\end{itemize}

\subsection{Proof of Theorem \ref{thm1}}\label{app:thm1}

For notational convenience, let $\Gamma_T$ denote the estimation error set of the membership set $\Theta_T$, i.e.,
\begin{equation*}
\Gamma_T := \Theta_T - \theta^* = \{(\hat{\theta}-\theta^*)\ :\ \hat{\theta}\in\Theta_T\}.
\end{equation*}
Since the translation by $-\theta^*$ is an isometry, we have $\diam(\Theta_T) = \diam(\Gamma_T)$.

In the following, we present a proof of Theorem \ref{thm1} in four steps.

\noindent \textbf{Step 1: Discussion on the persistent excitation.}

Since the PE condition is not assumed to always happen, but probabilistically, we need to consider two circumstances: (i) the PE condition is satisfied, and (ii) the PE condition is not satisfied. We define the following two events:
\begin{align*}
        &\mathcal{E}_1 := \left\{\exists\, \gamma\in\Gamma_T\ s.t.\ \vv \gamma\vv_F \geq\frac{\delta}{2}\right\},\\
        &\mathcal{E}_2 := \left\{\frac{1}{m}\sum_{s=1}^m z_{km+s}z^\top_{km+s}\succeq a_1^2 I_{n_z},\,\forall\, 0\leq k\leq \left\lceil\frac{T}{m}\right\rceil\!-\!1\right\},
\end{align*}
where $\mathcal{E}_1$ is the event that there exists a ``large error" in the estimation error set $\Gamma_T$. $\mathcal{E}_2$ is the event that the PE condition is satisfied. Next, we partition the event $\{\diam(\Theta_T)>\delta\}$ with $\mathcal{E}_1$ and $\mathcal{E}_2$. Namely, we can show that
\begin{equation*}
    \mathbb{P}\left(\diam(\Theta_T) > \delta\right)\leq \mathbb{P}\left(\mathcal{E}_1\cap\mathcal{E}_2\right) + \mathbb{P}\left(\mathcal{E}_2^\complement\right).
\end{equation*}
In other words, we cover the event $\{\diam(\Theta_T)>\delta\}$ by two events: (i) the event that the error estimation $\gamma$ is large and that the PE condition holds, and (ii) the event that the PE condition fails. To show this, we first prove that $\{\diam(\Theta_T)>\delta\}\subseteq\mathcal{E}_1$. To see this, we prove the contrapositive argument. Suppose $\mathcal{E}^\complement_1$ holds, i.e. $\forall\,\gamma\in\Gamma_T$, we have $\vv\gamma\vv_F < \frac{\delta}{2}$. It follows that $\forall\,\gamma_1,\gamma_2\in\Gamma_T$, we have $\vv\gamma_1 - \gamma_2\vv_F \leq \vv\gamma_1\vv_F + \vv\gamma_2\vv_F < \delta$. Taking supremum, we have $\diam(\Gamma_T) = \sup_{\gamma_1,\gamma_2\in\Gamma_T} \vv \gamma_1 - \gamma_2\vv_F \leq \delta$, which is $\{\diam(\Theta_T) > \delta\}^\complement$. It follows that $\mathcal{E}^\complement_1\subseteq \{\diam(\Theta_T) > \delta\}^\complement$. Thus, $\{\diam(\Theta_T) > \delta\}\subseteq\mathcal{E}_1$, and $\mathbb{P}(\diam(\Theta_T) > \delta) \leq \mathbb{P}(\mathcal{E}_1)$. Then, we show that $\mathbb{P}(\mathcal{E}_1)\leq\mathbb{P}(\mathcal{E}_1\cup\mathcal{E}_2) + \mathbb{P}(\mathcal{E}_2^\complement)$. To prove this, notice that $\mathcal{E}_1=(\mathcal{E}_1\cap\mathcal{E}_2) \cup (\mathcal{E}_1\cap\mathcal{E}_2^\complement)$. Moreover, since we have $\mathcal{E}_1\cap\mathcal{E}_2^\complement\subseteq \mathcal{E}_2^\complement$, it follows that $\mathcal{E}_1\subseteq (\mathcal{E}_1\cup\mathcal{E}_2)\cup\mathcal{E}_2^\complement$. By the sub-additivity of the probability measure, we have $\mathbb{P}(\diam(\Theta_T)>\delta) \leq \mathbb{P}(\mathcal{E}_1)\leq \mathbb{P}(\mathcal{E}_1\cup\mathcal{E}_2)+\mathbb{P}(\mathcal{E}_2^\complement)$.

Then, we aim to find an upper bound for each of the probabilities of $\mathcal{E}_1\cap\mathcal{E}_2$ and $\mathcal{E}^\complement_2$. These two upper bounds, when added up, serves as an upper bound for $\mathbb{P}(\diam(\Theta_T)>\delta$. It is shown in \cite{li2024setmembership} (Lemma 4.1) that
\begin{equation}\label{eq: lm41 in 14}
    \mathbb{P}(\mathcal{E}_2^\complement) \leq \textit{Term}\ 1\ in\ \eqref{eq4}.
\end{equation}
Therefore, to prove Theorem \ref{thm1}, we need to prove the following lemma.

\begin{lemma}\label{lm3}
If Assumptions \ref{ass:iid}, \ref{ass:bmsb}, and \ref{ass: tight bound} hold, then
\begin{equation*}
    \mathbb{P}\left(\mathcal{E}_1\cap\mathcal{E}_2\right)\leq \text{Term 2 in \eqref{eq4}}.
\end{equation*}
\end{lemma}

The rest of Appendix \ref{app:thm1} focuses on the proof of Lemma \ref{lm3}. Firstly, to prove Lemma \ref{lm3}, we employ a spatial discretization of the event $\mathcal{E}_1\cap\mathcal{E}_2$ (Step 2). Secondly, we employ a temporal discretization with respect to the PE window $m$ upon the previous discretization (Step 3). Finally, we combine the two discretizations and show the upper bound in Lemma \ref{lm3}.

\noindent\textbf{Step 2: Discretization on the space.}
Namely, we cover the event $\mathcal{E}_1\cap\mathcal{E}_2$ with a finite collection of sub-events, denoted $\{\tilde{\mathcal{E}}_{1,i}\}_i$, which is defined later in Claim \ref{claim4}. Before introducing these sub-events, we require a small-ball covering of the Frobenius-norm unit sphere, a stopping time, and an adapted process, which will be defined in the following paragraphs.

Firstly, we consider a small ball covering of the unit $n_x\times n_z$-dimensional unit Frobenius-norm sphere $\mathbb{S}^F_1(0):=\{\bar\gamma\in\mathbb{R}^{n_x\times n_z}\ :\ \vv \bar\gamma\vv_F = 1\}$. We consider covering $\mathbb{S}_1^F(0)$ by smaller balls with radius $\tilde{\epsilon}_\gamma = \frac{1}{a_4} = \min\left\{1, \frac{\sigma_z p_z}{16 b_z}\right\}$ where $\sigma_z,p_z,b_z$ are defined in Assumption \ref{ass:bmsb}, and denote the set of covering $\tilde{\epsilon}_\gamma$-balls to be $\{\tilde{\mathcal{B}}_{i}\}_{i=1}^{\Omega_\gamma}$ (i.e. $\bigcup_{i=1}^{\Omega_\gamma} \tilde{\mathcal{B}}_{i} \supseteq \mathbb{S}_1^F(0)$). Here, $\Omega_\gamma$ is the number of small $\tilde{\epsilon}_\gamma$-balls required to cover the sphere. Let $\tilde{\mathcal{M}} = \{\tilde{\gamma}_i\}_{i=1}^{\Omega_\gamma}$ such that $\forall\,i\in\{1,\cdots, \Omega_\gamma\}$, $\tilde{\gamma}_i\in\mathbb{S}_1^F(0)\cap\tilde{\mathcal{B}}_i$ be a $\tilde{\epsilon}_\gamma$-mesh of $\mathbb{S}_1^F(0)$. Thus, we have $\forall\, \bar\gamma\in\mathbb{S}^F_1(0)$, $\exists \tilde{\gamma}_i\in\tilde{\mathcal{M}}$ such that $\vv \tilde{\gamma}_i - \bar\gamma\vv_F \leq 2 \tilde{\epsilon}_\gamma$. For the theory of covering number, we refer the reader to \cite{c10}, \cite{c11} and Appendix D.1 in \cite{li2024setmembership}. Here we have:
\begin{equation}\label{eq: covering1}
    \Omega_\gamma \leq \Tilde{O}\left((n_xn_z)^{5/2}\right)a_4^{n_xn_z}.
\end{equation}
Secondly, define a stopping time to be used in Claim \ref{claim4}.
\begin{equation*}
    \tilde{L}_{i,k}:=\min\left\{m+1,\min\left\{l\geq 1\ :\ \vv\tilde{\gamma}_iz_{km+l}\vv_2\geq a_1\right\}\right\}.
\end{equation*}
Thirdly, $\forall\, t\geq 0$, define a process $\{\tilde{v}_{i,t}\}_{t\geq 0}$ adapted to the natural filtration of the process $\{z_t\}_{t\geq 0}$
\begin{equation*}
    \tilde{v}_{i,t} := \arg\max_{v\in \mathbb{S}^2_1(0)} v^\top (\tilde{\gamma}_i z_t),
\end{equation*}
where $\mathbb{S}^2_1(0):=\{v\in\mathbb{R}^{n_x}\ :\ \vv v\vv_2 = 1\}$. 

Next, with the help of $\tilde{L}_{i,k}$ and $\tilde{v}_{i,t}$, we cover the event $\mathcal{E}_1\cap\mathcal{E}_2$ with sub-events with respect to the points in the $\frac{1}{a_4}$-net of $\mathbb{S}_1^F(0)$ with the following claim.

\noindent
\begin{claim}\label{claim4}
    $\forall\, i\in\{1,\cdots, \Omega_\gamma\}$, define
\begin{align*}
    \tilde{\mathcal{E}}_{1,i}:=&\left\{\exists\, \gamma\in\Gamma_T\ \text{such that}\ \forall\, k\in\left\{0,\cdots, \left\lceil \frac{T}{m}\right\rceil -1\right\},\right.\\
    &\left.\tilde{v}^\top_{i,km+\tilde{L}_{i,k}}(\gamma z_{km+\tilde{L}_{i,k}})\geq\frac{a_1\delta}{4}\right\}\cap\mathcal{E}_2.
\end{align*}
Then
\begin{equation}\label{eq: events covered}
\mathcal{E}_1\cap\mathcal{E}_2\subseteq \bigcup_{i=1}^{\Omega_\gamma}\tilde{\mathcal{E}}_{1,i}.
\end{equation}
Thus,
\begin{equation}\label{ineq: claim11}
     \mathbb{P}(\mathcal{E}_1\cap\mathcal{E}_2)\leq \Tilde{O}\left((n_xn_z)^{5/2}\right)a_4^{n_xn_z}\max_{1\leq i\leq \Omega_\gamma}\mathbb{P}(\tilde{\mathcal{E}}_{1,i}).
\end{equation}
\end{claim}

\noindent
\begin{proof}[Proof of Claim \ref{claim4}] 

We first prove \eqref{eq: events covered} by showing that the event $\mathcal{E}_1\cap\mathcal{E}_2$ implies $\bigcup_{i=1}^{\Omega_\gamma}\tilde{\mathcal{E}}_{1,i}$. We then prove \eqref{ineq: claim11} by the sub-additivity of the probability measure.

To prove \eqref{eq: events covered}, given $\mathcal{E}_1\cap\mathcal{E}_2$ holds, i.e. $\exists\,\gamma\in\Gamma_T$ such that $\vv\gamma\vv_F\geq\frac{\delta}{2}$, and $\forall\,k\in\left\{0,1,\cdots,\left\lceil\frac{T}{m}\right\rceil - 1\right\}$, $\frac{1}{m}\sum_{s=1}^mz_{km+s}z^\top_{km+s}\succeq a_1^2 I_{n_z}$. Consider unit directional vector of $\gamma$ denoted by $\bar{\gamma}$:
\begin{equation*}
    \bar\gamma := \frac{\gamma}{\vv\gamma\vv_F} \in\mathbb{S}_1^F(0).
\end{equation*}
By the small ball covering theory (\cite{c10}, \cite{c11} and Appendix D.1 in \cite{li2024setmembership}), $\bar\gamma$ must be close to some point in the $\frac{1}{a_4}$-mesh, i.e. $\exists\, \tilde{\gamma}_i\in\tilde{\mathcal{M}}$ such that $\vv\bar\gamma-\tilde{\gamma}_i\vv_F\leq\frac{2}{a_4}$. Meanwhile, since $\forall\,k\in\left\{0,1,\cdots,\left\lceil\frac{T}{m}\right\rceil - 1\right\}$, $\frac{1}{m}\sum_{s=1}^mz_{km+s}z^\top_{km+s}\succeq a_1^2 I_{n_z}$, by multiplying $\tilde{\gamma}_i$ on the left, multiplying $\tilde{\gamma}_i^\top$ on the right, and then taking matrix traces, we have
\begin{equation*}
    \frac{1}{m}\sum_{s=1}^m \vv\tilde{\gamma}_i z_{km+s}\vv_2^2\geq a_1^2.
\end{equation*}
Therefore, 
by the Pigeonhole Principle, for any $i,k$, there  exists $s\in\{1,\cdots,m\}$ such that $\vv\tilde{\gamma}_i z_{km+s}\vv_2\geq a_1$. Notice that the least of such $s$ is indeed the stopping time $\tilde{L}_{i,k}$. This also implies that $\forall\,i,k$, $\tilde{L}_{i,k}\leq m$.
It follows that
\begin{align}\label{ineq:claim1}
    &\tilde{v}^\top_{i,km+\tilde{L}_{i,k}}(\bar\gamma z_{km+\tilde{L}_{i,k}})\nonumber\\
    =&\tilde{v}^\top_{i,km+\tilde{L}_{i,k}}\left\{\left[\tilde{\gamma}_i - (\tilde{\gamma}_i -\bar\gamma)\right]z_{km+\tilde{L}_{i,k}}\right\}\nonumber\\
    =&\tilde{v}^\top_{i,km+\tilde{L}_{i,k}}(\tilde{\gamma}_i z_{km+\tilde{L}_{i,k}}) - \tilde{v}^\top_{i,km+\tilde{L}_{i,k}}\left[(\tilde{\gamma}_i-\bar\gamma)z_{km+\tilde{L}_{i,k}}\right]\nonumber\\
    \geq&a_1 - \vv\tilde{\gamma}_i -\bar\gamma\vv_2\cdot\vv z_{km+\tilde{L}_{i,k}}\vv_2\nonumber\\
    \geq&a_1 - \frac{2b_z}{a_4}\nonumber\\
    \geq&\frac{a_1}{2}.
\end{align}
Recall that we can rewrite $\gamma$ as its norm multiplying its unit directional vector, i.e., $\gamma = \bar\gamma\vv\gamma\vv_F$. By multiplying \eqref{ineq:claim1} by $\vv\gamma\vv_F$, and by that $\vv\gamma\vv_F \geq \frac{\delta}{2}$, we obtain that there exists $i\in\{1,\cdots, \Omega_\gamma\}$ such that $\forall k\in\left\{0,\cdots, \left\lceil\frac{T}{m}\right\rceil-1\right\}$
\begin{equation}\label{eq:claim1}
    \tilde{v}^\top_{i,km+\tilde{L}_{i,k}}(\gamma z_{km+\tilde{L}_{i,k}})\geq\frac{a_1\delta}{4}.
\end{equation}
It follows that $\mathcal{E}_1\cap\mathcal{E}_2$ implies $\bigcup_{i=1}^{\Omega_\gamma}\tilde{\mathcal{E}}_{1,i}$, which is \eqref{eq: events covered}.

Secondly, we prove \eqref{ineq: claim11}. Since we have $\mathcal{E}_1\cap\mathcal{E}_2\subseteq \bigcup_{i=1}^{\Omega_\gamma}\tilde{\mathcal{E}}_{1,i}$, then, by the sub-additivity of the probability measure, we have
\begin{equation*}
    \mathbb{P}\left(\mathcal{E}_1\cap\mathcal{E}_2\right) \leq \sum_{i=1}^{\Omega_\gamma}\mathbb{P}\left(\tilde{\mathcal{E}}_{1,i}\right)\leq \Omega_\gamma\max_{i\in\{1,\cdots,\Omega_\gamma\}}\mathbb{P}(\tilde{\mathcal{E}}_{1,i}).
\end{equation*}
By \eqref{eq: covering1}, the above inequality implies \eqref{ineq: claim11}, which concludes this proof.
\end{proof}

Based on Claim \ref{claim4}, we aim to find a uniform upper bound for $\left\{\mathbb{P}\left(\tilde{\mathcal{E}}_{1,i}\right)\right\}_{i=1}^{\Omega_\gamma}$. To address this problem, we provide a temporal discretization with respect to the $k$ time segments with length $m$.

\noindent\textbf{Step 3: Discretization on time.} Specifically, we partition the time $[0,T]$ into segments with length $m$, and construct one event per time segment such that the intersection of all these events form a superset of $\tilde{\mathcal{E}}_{1,i}$. In the following claim, we prove that $\tilde{\mathcal{E}}_{1,i}$ can be covered by a sequence of subevents $\{\tilde{G}_{i,k}\}_{k=0}^{\left\lceil\frac{T}{m}\right\rceil - 1}$ to be defined below.
\begin{claim}\label{claim5}
    $\forall\, i\in\{1,\cdots,\Omega_\gamma\}, k\in\left\{0,\cdots,\left\lceil \frac{T}{m}\right\rceil - 1\right\}$, define
\begin{equation*}
    w^0_{i,km+\tilde{L}_{i,k}}:= \arg\min_{w\in\mathbb{W}}w^\top\tilde{v}_{i,km+\tilde{L}_{i,k}}.
\end{equation*}
and
\begin{equation*}
    \tilde{G}_{i,k}:=\left\{\vv w_{km+\tilde{L}_{i,k}}-w^0_{i,km+\tilde{L}_{i,k}}\vv_2\geq\frac{a_1\delta}{4}\right\}\cap\mathcal{E}_2.
\end{equation*}
Then $\forall\, i\in\{1,\cdots,\Omega_\gamma\}$, we have
    \begin{equation}\label{eq: claim12}
    \tilde{\mathcal{E}}_{1,i}\subseteq \bigcap_{k=0}^{\left\lceil \frac{T}{m}\right\rceil - 1}\tilde{G}_{i,k}.
    \end{equation}
\end{claim}

\noindent
\begin{proof}[Proof of Claim \ref{claim5}]
    Recall that by the algorithm \eqref{equ: SME algo}, $\forall\, t\geq 0$,
\begin{equation*}
    w_t-\gamma z_t = x_{t+1}-\hat\theta z_t\in\mathbb{W}.
\end{equation*}
Since $\mathbb{W}$ is convex, compact, and has a non-empty interior, by the Supporting Hyperplane Theorem, if we denote
\begin{equation*}
    h(v):= \min_{w\in\mathbb{W}} v^\top w,
\end{equation*}
then we can represent $\mathbb{W}$ in the following manner.
\begin{equation*}
    \mathbb{W} = \{w\in\mathbb{R}^{n_x}\ :\ \forall\, v\in\mathbb{S}_1^2(0),\ v^\top w\geq h(v)\}.
\end{equation*}
Therefore, $\forall v\in\mathbb{S}_1^2(0)$, $v^\top(w_t - \gamma z_t)\geq h(v)$. Substituting $v$ by $\tilde{v}_{i,km+\tilde{L}_{i,k}}$, we have
\begin{equation*}
    \tilde{v}^\top_{i,km+\tilde{L}_{i,k}}(w_{km+\tilde{L}_{i,k}} -\gamma z_{km+\tilde{L}_{i,k}})\geq h(\tilde{v}_{i,km+\tilde{L}_{i,k}}).
\end{equation*}
It follows that, when $\tilde{\mathcal{E}}_{1,i}$ holds, by \eqref{eq:claim1}, we have
\begin{align}\label{ineq: claim2}
    &\tilde{v}^\top_{i,km+\tilde{L}_{i,k}}w_{km+\tilde{L}_{i,k}}-h(\tilde{v}_{i,km+\tilde{L}_{i,k}})\nonumber\\
    \geq &\tilde{v}^\top_{i,km+\tilde{L}_{i,k}}(\gamma z_{km+\tilde{L}_{i,k}})\nonumber\\
    \geq &\frac{a_1\delta}{4}.
\end{align}
Since $\mathbb{W}$ is compact and $v^\top w$ is a linear functional of $w$ with any fixed $v$, then there exists a minimizer for $v^\top w$ among the boundary points of $\mathbb{W}$ (i.e., $\exists\,w^*\in\W$ such that $h(v) = v^\top w^*$). For the case of $v=\tilde{v}_{i,km+\tilde{L}_{i,k}}$, we denote such a minimizer by $w^0_{i,km+\tilde{L}_{i,k}}$. That is, for all $i\in\{1,\cdots, \Omega_\gamma\}$, $k\in\left\{0,\cdots,\left\lceil \frac{T}{m}\right\rceil - 1\right\}$,
\begin{equation}\label{eq: claim2}
    w^0_{i,km+\tilde{L}_{i,k}}\in \left(\arg\min_{w\in\mathbb{W}} w^\top \tilde{v}_{i,km+\tilde{L}_{i,k}}\cap\partial\mathbb{W}\right)\neq\emptyset.
\end{equation}
Though such a minimizer may not be unique, yet we can take any one of them, as long as it satisfies \eqref{eq: claim2}. It follows that
\begin{equation*}
    h(\tilde{v}_{i,km+\tilde{L}_{i,k}}) = \tilde{v}^\top_{i,km+\tilde{L}_{i,k}}w^0_{i,km+\tilde{L}_{i,k}}.
\end{equation*}
This implies
\begin{align*}
    \frac{a_1\delta}{4} &\overset{(a)}{\leq} |\tilde{v}^\top_{i,km+\tilde{L}_{i,k}}(w_{km+\tilde{L}_{i,k}} - w^0_{i,km+\tilde{L}_{i,k}})|\\
    &\overset{(b)}{\leq} \vv \tilde{v}_{i,km+\tilde{L}_{i,k}}\vv_2\cdot\vv w_{km+\tilde{L}_{i,k}} - w^0_{i,km+\tilde{L}_{i,k}}\vv_2\\
    &\overset{(c)}{\leq} \vv w_{km+\tilde{L}_{i,k}} - w^0_{i,km+\tilde{L}_{i,k}}\vv_2,
\end{align*}
where (a) is deduced by \eqref{ineq: claim2} and \eqref{eq: claim2}; (b) is a direct application of the Cauchy-Schwartz Inequality; (c) is implied by the fact that $\vv \tilde{v}_{i,km+\tilde{L}_{i,k}}\vv_2 = 1$. Thus, if we define events $\{\tilde{G}_{i,k}\}_{k=0}^{\lceil\frac{T}{m}\rceil-1}$ in the following manner
\begin{equation*}
    \tilde{G}_{i,k}:=\left\{\vv w_{km+\tilde{L}_{i,k}}-w^0_{i,km+\tilde{L}_{i,k}}\vv_2\geq\frac{a_1\delta}{4}\right\}\cap\mathcal{E}_2,
\end{equation*}
then we can deduce that $\forall\, i\in\{1,\cdots,\Omega_\gamma\}$, we have
\begin{equation*}
    \tilde{\mathcal{E}}_{1,i}\subseteq \bigcap_{k=0}^{\left\lceil \frac{T}{m}\right\rceil - 1} \tilde{G}_{i,k}.
\end{equation*}
\end{proof}
With Claim \ref{claim5} proved, we provide an upper bound for $\mathbb{P}(\tilde{\mathcal{E}}_{1,i})$ by providing an upper bound for $\mathbb{P}(\bigcap_{k=0}^{\left\lceil\frac{T}{m}\right\rceil - 1}\tilde{G}_{i,k})$ with Claim \ref{claim6}.
\noindent
\begin{claim}\label{claim6}
    For any $i\in\{1,\cdots, \Omega_\gamma\}$:
\begin{equation*}
    \mathbb{P}\left(\bigcap_{k=0}^{\left\lceil \frac{T}{m}\right\rceil - 1} \tilde{G}_{i,k}\right)\leq \left(1 - q_w\left(\frac{a_1\delta}{4}\right)\right)^{\left\lceil\frac{T}{m}\right\rceil - 1}.
\end{equation*}
\end{claim}

\begin{proof}[Proof of Claim \ref{claim6}]
    By Bayes' Theorem, $\forall\, i$:
\begin{equation*}
    \mathbb{P}\left(\bigcap_{k=0}^{\left\lceil \frac{T}{m}\right\rceil - 1} \tilde{G}_{i,k}\right) = \mathbb{P}\left(\tilde{G}_{i,0}\right)\displaystyle\prod_{k=1}^{\left\lceil \frac{T}{m}\right\rceil - 1}\mathbb{P}\left(\tilde{G}_{i,k}\middle| \bigcap_{\ell=0}^{k-1}\tilde{G}_{i,\ell}\right).
\end{equation*}
Recall that $\mathcal{E}_2\subseteq \left\{1\leq \tilde{L}_{i,k} \leq m,\ \forall\,i,k\right\}$. Hence, $\tilde{G}_{i,k}\subseteq \left\{\vv w_{km+\tilde{L}_{i,k}}-w^0_{i,km+\tilde{L}_{i,k}}\vv_2\geq\frac{a_1\delta}{4}\text{ and }1\leq \tilde{L}_{i,k} \leq m\right\}$. For any $k\in\left\{2,3,\cdots,\left\lceil\frac{T}{m}\right\rceil - 1\right\}$, we have

\begin{align*}
    &\mathbb{P}\left(\tilde{G}_{i,k}\middle|\bigcap_{j=0}^{k-1}\tilde{G}_{i,j}\right)\\
    \leq&\mathbb{P}\left(\vv w_{km+\tilde{L}_{i,k}} - w^0_{i,km+\tilde{L}_{i,k}}\vv_2\geq\frac{a_1\delta}{4};\right.\\
    &\left. 1\leq \tilde{L}_{i,k}\leq m\middle| \bigcap_{j=0}^{k-1}\tilde{G}_{i,j}\right)\\
    =&\sum_{l=1}^m\!\mathbb{P}\!\left(\!\vv w_{km+l} - w^0_{i,km+l}\vv_2\geq\frac{a_1\delta}{4};\ \tilde{L}_{i,k}=l \middle| \bigcap_{j=0}^{k-1}\tilde{G}_{i,j}\right)\\
    = &\sum_{l=1}^m\left[\mathbb{P}\left(\vv w_{km+l} - w^0_{i,km+l}\vv_2\geq\frac{a_1\delta}{4}\middle| \tilde{L}_{i,k}=l, \bigcap_{j=0}^{k-1}\tilde{G}_{i,j}\right)\right.\\
    &\left.\times\mathbb{P}\left(\tilde{L}_{i,k}=l\middle| \bigcap_{j=0}^{k-1}\tilde{G}_{i,j}\right)\right].
\end{align*}
The first equation is deduced by the law of total probability, and the second equation is deduced by Bayes' theorem. Meanwhile, notice that
\begin{align*}
    &\mathbb{P}\left(\vv w_{km+l} - w^0_{i,km+l}\vv_2\geq\frac{a_1\delta}{4}\middle| \tilde{L}_{i,k}=l, \bigcap_{j=0}^{k-1}\tilde{G}_{i,j}\right)\\
    \overset{(d)}{=}&\int_{v_{0:km+l}}\mathbb{P}\left(\vv w_{km+l}-w^0_{i,km+l}\vv_2\geq\frac{a_1\delta}{4},\right.\\
    &\left.w_{0:km+l}=v_{0:km+l}\middle| \tilde{L}_{i,k}=l,\bigcap_{j=0}^{k-1}\tilde{G}_{i,j}\right)dv_{0:km+l}\\
    \overset{(e)}{=}&\int_{\tilde{Q}_{k,l}}\!\left[\!\mathbb{P}\!\left(\!\vv w_{km+l}-w^0_{i,km+l}\vv_2\!\geq\!\frac{a_1\delta}{4}\middle|\! w_{0:km+l}=v_{0:km+l}\right)\right.\\
    &\left.\times \mathbb{P}\left(w_{0:km+l}=v_{0:km+l}\middle| \tilde{L}_{i,k}=l,\bigcap_{j=0}^{k-1}\tilde{G}_{i,j}\right)\right]dv_{0:km+l}\\
    \leq&\left(1-q_w\left(\frac{a_1\delta}{4}\right)\right)\times\\
    &\int_{\tilde{Q}_{k,l}}\mathbb{P}\left(w_{0:km+l}=v_{0:km+l}\middle| \tilde{L}_{i,k}=l,\bigcap_{j=0}^{k-1}\tilde{G}_{i,j}\right)dv_{0:km+l}\\
    =&1-q_w\left(\frac{a_1\delta}{4}\right),
\end{align*}
where (d) is deduced by the law of total probability, and (e) is deduced by Bayes' theorem. Specially, we denote a sequence of noise (or its realization) by
\begin{equation*}
    w_{0:km+l}:=\{w_0,w_1,\cdots, w_{km+l-1}\},
\end{equation*}
and the domain of integration by
\begin{align*}
    \tilde{Q}_{k,l}:=&\{v_{0:km+l}\text{ such that }w_{0:km+l}=v_{0:km+l}\text{ satisfies }\\
    &\tilde{L}_{i,k}=l\text{ and }\bigcap_{j=1}^{k-1}\tilde{G}_{i,j}\}.
\end{align*}
This follows that
\begin{align*}
    &\mathbb{P}\left(\tilde{G}_{i,k}\middle|\bigcap_{j=0}^{k-1}\tilde{G}_{i,j}\right)\\
    \leq&\left(1-q_w\left(\frac{a_1\delta}{4}\right)\right)\sum_{l=1}^m\mathbb{P}\left(\tilde{L}_{i,k}=l\middle|\bigcap_{j=0}^{k-1}\tilde{G}_{i,j}\right)\\
    =&1-q_w\left(\frac{a_1\delta}{4}\right).
\end{align*}
By a similar procedure, we can also deduce that
\begin{equation*}
    \mathbb{P}\left(\tilde{G}_{i,0}\right) \leq 1-q_w\left(\frac{a_1\delta}{4}\right).
\end{equation*}
It follows that 
\begin{equation}\label{ineq: claim22}
    \mathbb{P}\left(\bigcap_{k=0}^{\left\lceil \frac{T}{m}\right\rceil - 1} \tilde{G}_{i,k}\right)\leq \left(1 - q_w\left(\frac{a_1\delta}{4}\right)\right)^{\left\lceil\frac{T}{m}\right\rceil - 1}.
\end{equation}
\end{proof}

\noindent
\textbf{Step 4: The final assembly.}
Combining \eqref{eq: events covered}, \eqref{eq: claim12}, and \eqref{ineq: claim22}, we can prove Lemma \ref{lm3}. Combining Lemma \ref{lm3} with \ref{eq: lm41 in 14}, we prove Theorem \ref{thm1}.

\subsection{Proof of Theorem \ref{thm2}}\label{app: thm2}
We present a proof of Theorem \ref{thm2} that follows a similar scheme as the proof of Theorem \ref{thm1} in Appendix \ref{app:thm1} except (i) we use the supporting half-spaces defined in Definition \ref{def: shs} to construct events $\{G_{i,k}\}_{k=0}^{\left\lceil\frac{T}{m}\right\rceil}$ in Claim \ref{claim5_tl}; and (ii) we use the thin-slice boundary visiting assumption (Assumption \ref{ass: tbtl}) to find an upper bound for the 
probability of $\bigcap_{k=0}^{\left\lceil\frac{T}{m}\right\rceil - 1}G_{i,k}$ in Claim \ref{claim6_tl}. Recall that we let $\Gamma_T$ denote the estimation error set of the membership set 
$\Theta_T$, i.e. $\Gamma_T := \Theta_T - \theta^* = \{(\hat{\theta}-\theta^*)\ :\ \hat{\theta}\in\Theta_T\}$, and we have $\diam(\Theta_T) = \diam(\Gamma_T)$. In the following, we present a proof of Theorem \ref{thm2} in four steps.

\noindent \textbf{Step 1: Discussion on the persistent excitation.}

Since the PE condition is not assumed to always happen, but probabilistically, we need to consider two circumstances: (i) the PE condition is satisfied, and (ii) the PE condition is not satisfied. We define the following two events:
\begin{align*}
        &\mathcal{E}_1 := \left\{\exists\, \gamma\in\Gamma_T\ s.t.\ \vv \gamma\vv_F \geq\frac{\delta}{2}\right\},\\
        &\mathcal{E}_2 := \left\{\frac{1}{m}\sum_{s=1}^m z_{km+s}z^\top_{km+s}\succeq a_1^2 I_{n_z},\,\forall\, 0\leq k\leq \left\lceil\frac{T}{m}\right\rceil\!-\!1\right\}.
\end{align*}

With exactly the same reasoning as in Step 1 of Appendix \ref{app:thm1}, we can show that
\begin{equation}\label{eq:split}
    \mathbb{P}\left(\diam(\Theta_T) > \delta\right)\leq \mathbb{P}\left(\mathcal{E}_1\cap\mathcal{E}_2\right) + \mathbb{P}\left(\mathcal{E}_2^\complement\right).
\end{equation}
In other words, we cover the event $\{\diam(\Theta_T)>\delta\}$ by two events: (i) the event that the error is large and that the PE condition holds, and (ii) the event that the PE condition fails.

Then, we aim to find an upper bound for each of the probabilities of the two events. These two upper bounds, when added up, serves as an upper bound for $\mathbb{P}(\diam(\Theta_T)>\delta$. It is shown in \cite{li2024setmembership} (Lemma 4.1) that $\mathbb{P}(\mathcal{E}_2^\complement) \leq \textit{Term}\ 1\ in\ \eqref{eqtl}$. Therefore, to prove Theorem \ref{thm2}, what it remains to be shown is the following lemma.

\begin{lemma}\label{lm3_tl}
If Assumptions \ref{ass:iid}, \ref{ass:bmsb}, and \ref{ass: tbtl} hold, then
\begin{equation*}
    \mathbb{P}\left(\mathcal{E}_1\cap\mathcal{E}_2\right)\leq \text{Term 3 in \eqref{eqtl}}.
\end{equation*}
\end{lemma}

The remaining of this proof focuses on showing Lemma \ref{lm3_tl}. Firstly, Step 2 provides a spatial discretization of the event $\mathcal{E}_1\cap\mathcal{E}_2$. Secondly, Step 3 provides a temporal discretization with respect to the PE window $m$ upon the previous step. Finally, we combine the two discretizations and show the upper bound in Lemma \ref{lm3_tl}.

\noindent\textbf{Step 2: Discretization on the space.}
Namely, we cover the event $\mathcal{E}_1\cap\mathcal{E}_2$ with a finite collection of sub-events denoted $\{\mathcal{E}_{1,i}\}_i$, which is defined later in Claim \ref{claim4_tl}.

Before introducing the sub-events, recall the small ball covering theory introduced in \cite{c10}, \cite{c11} and Appendix D.1 in \cite{li2024setmembership}. We consider covering the unit $n_x\times n_z$-dimensional unit Frobenius-norm sphere $\mathbb{S}_1^F(0)$ by smaller balls with radius $\epsilon_\gamma = \frac{1}{a_5} = \min\left\{1, \frac{\sigma_z p_z\xi}{16 b_z}\right\}$ where $\sigma_z,p_z,b_z$ are defined in Assumption \ref{ass:bmsb}, and denote the set of covering $\epsilon_\gamma$-balls to be $\{\mathcal{B}_{i}\}_{i=1}^{\Delta_\gamma}$ (i.e. $\bigcup_{i=1}^{\Delta_\gamma} \mathcal{B}_{i} \supseteq \mathbb{S}_1^F(0)$), where $\Delta_\gamma$ is the number of small $\epsilon_\gamma$-balls required to cover the sphere. Let $\mathcal{M} = \{\gamma_i\}_{i=1}^{\Delta_\gamma}$ such that $\forall\,i\in\{1,\cdots, \Delta_\gamma\}$, $\gamma_i\in\mathbb{S}_1^F(0)\cap\mathcal{B}_i$ be a $\epsilon_\gamma$-mesh of $\mathbb{S}_1^F(0)$. Thus, we have $\forall\, \bar\gamma\in\mathbb{S}^F_1(0)$, $\exists \gamma_i\in\mathcal{M}$ such that $\vv \gamma_i - \bar\gamma\vv_F \leq 2 \epsilon_\gamma$. Then we have:
\begin{equation}\label{eq: covering}
    \Delta_\gamma \leq \Tilde{O}\left((n_xn_z)^{5/2}\right)a_5^{n_xn_z}.
\end{equation}
Secondly, define a stopping time
\begin{equation*}
    L_{i,k}:=\min\left\{m+1,\min\left\{l\geq 1\ :\ \vv\gamma_iz_{km+l}\vv_2\geq a_1\right\}\right\}.
\end{equation*}
Thirdly, $\forall\, t\geq 0$, define two adapted processes $\{h_{i,t}\}_{t\geq 0}$ and $\{h_{i,t}\}_{t\geq 0}$.
\begin{equation*}
    c_{i,t}, h_{i,t} := \arg\max_{H(c,h)\in\{H(c_j,h_j)\}_{j=1}^{\mathcal{X}}} h^\top (\gamma_i z_t).
\end{equation*}

Next, we cover the event $\mathcal{E}_1\cap\mathcal{E}_2$ with sub-events with respect to the points in the $\frac{1}{a_5}$-net of $\mathbb{S}_1^F(0)$. Namely, we have the following claim.

\noindent
\begin{claim}\label{claim4_tl}
    $\forall\, i\in\{1,\cdots, \Delta_\gamma\}$, define
\begin{align*}
    \mathcal{E}_{1,i}:=&\left\{\exists\, \gamma\in\Gamma_T\ \text{such that}\ \forall\, k\in\left\{0,\cdots, \left\lceil \frac{T}{m}\right\rceil -1\right\},\right.\\
    &\left.h^\top_{i,km+L_{i,k}}(\gamma z_{km+L_{i,k}})\geq\frac{a_1\delta\xi}{4}\right\}\cap\mathcal{E}_2.
\end{align*}
Then
\begin{equation}\label{eq: events covered_tl}
\mathcal{E}_1\cap\mathcal{E}_2\subseteq \bigcup_{i=1}^{\Delta_\gamma}\mathcal{E}_{1,i}.
\end{equation}
Thus,
\begin{equation}\label{ineq: claim11_tl}
     \mathbb{P}(\mathcal{E}_1\cap\mathcal{E}_2)\leq \Tilde{O}\left((n_xn_z)^{5/2}\right)a_5^{n_xn_z}\max_{1\leq i\leq \Delta_\gamma}\mathbb{P}(\mathcal{E}_{1,i}).
\end{equation}
\end{claim}

\noindent
\begin{proof}[Proof of Claim \ref{claim4_tl}]
We first prove \eqref{eq: events covered_tl} by showing that the event $\mathcal{E}_1\cap\mathcal{E}_2$ implies $\bigcup_{i=1}^{\Delta_\gamma}\mathcal{E}_{1,i}$. We then prove \eqref{ineq: claim11_tl} by the sub-additivity of the probability measure.

To prove \eqref{eq: events covered_tl}, given $\mathcal{E}_1\cap\mathcal{E}_2$ holds, i.e. $\exists\,\gamma\in\Gamma_T$ such that $\vv\gamma\vv_F\geq\frac{\delta}{2}$, and $\forall\,k\in\left\{0,1,\cdots,\left\lceil\frac{T}{m}\right\rceil - 1\right\}$, $\frac{1}{m}\sum_{s=1}^mz_{km+s}z^\top_{km+s}\succeq a_1^2 I_{n_z}$. Consider unit directional vector of $\gamma$ denoted by $\bar{\gamma}$:
\begin{equation*}
    \bar\gamma := \frac{\gamma}{\vv\gamma\vv_F} \in\mathbb{S}_1^F(0).
\end{equation*}
By the small ball covering theory (\cite{c10}, \cite{c11} and Appendix D.1 in \cite{li2024setmembership}), $\bar\gamma$ must be close to some point in the $\frac{1}{a_5}$-mesh, i.e. $\exists\, \gamma_i\in\mathcal{M}$ such that $\vv\bar\gamma-\gamma_i\vv_F\leq\frac{2}{a_5}$. Meanwhile, since $\forall\,k\in\left\{0,1,\cdots,\left\lceil\frac{T}{m}\right\rceil - 1\right\}$, $\frac{1}{m}\sum_{s=1}^mz_{km+s}z^\top_{km+s}\succeq a_1^2 I_{n_z}$, by multiplying $\gamma_i$ on the left, multiplying $\gamma_i^\top$ on the right, and then taking matrix traces, we have
\begin{equation*}
    \frac{1}{m}\sum_{s=1}^m \vv\gamma_i z_{km+s}\vv_2^2\geq a_1^2.
\end{equation*}
Therefore, 
by the Pigeonhole Principle, for any $i,k$, there  exists $s\in\{1,\cdots,m\}$ such that $\vv\gamma_i z_{km+s}\vv_2\geq a_1$. Notice that the least of such $s$ is indeed the stopping time $L_{i,k}$. This also implies that $\forall\,i,k$, $L_{i,k}\leq m$.
It follows that
\begin{align}\label{ineq:claim1_tl}
    &h^\top_{i,km+L_{i,k}}(\bar\gamma z_{km+L_{i,k}})\nonumber\\
    =&h^\top_{i,km+L_{i,k}}\left\{\left[\gamma_i - (\gamma_i -\bar\gamma)\right]z_{km+L_{i,k}}\right\}\nonumber\\
    =&h^\top_{i,km+L_{i,k}}(\gamma_i z_{km+L_{i,k}}) - h^\top_{i,km+L_{i,k}}\left[(\gamma_i-\bar\gamma)z_{km+L_{i,k}}\right]\nonumber\\
    \geq&a_1\xi - \vv\gamma_i -\bar\gamma\vv_2\cdot\vv z_{km+L_{i,k}}\vv_2\nonumber\\
    \geq&a_1\xi - \frac{2b_z}{a_5}\nonumber\\
    \geq&\frac{a_1\xi}{2}.
\end{align}
Recall that we can rewrite $\gamma$ as its norm multiplying its unit directional vector, i.e., $\gamma = \vv\gamma\vv_F\bar\gamma$. By multiplying \eqref{ineq:claim1_tl} by $\vv\gamma\vv_F$, and by that $\vv\gamma\vv_F \geq \frac{\delta}{2}$, we obtain that there exists $i\in\{1,\cdots, \Delta_\gamma\}$ such that $\forall k\in\left\{0,\cdots, \left\lceil\frac{T}{m}\right\rceil-1\right\}$
\begin{equation}\label{eq:claim1_tl}
    h^\top_{i,km+L_{i,k}}(\gamma z_{km+L_{i,k}})\geq\frac{a_1\delta\xi}{4}.
\end{equation}
It follows that $\mathcal{E}_1\cap\mathcal{E}_2$ implies $\bigcup_{i=1}^{\Delta_\gamma}\mathcal{E}_{1,i}$, which is \eqref{eq: events covered_tl}. 


Secondly, we prove \eqref{ineq: claim11_tl}. Since we have $\mathcal{E}_1\cap\mathcal{E}_2\subseteq \bigcup_{i=1}^{\Delta_\gamma}\mathcal{E}_{1,i}$, then, by the sub-additivity of the probability measure, we have
\begin{equation*}
    \mathbb{P}\left(\mathcal{E}_1\cap\mathcal{E}_2\right) \leq \sum_{i=1}^{\Delta_\gamma}\mathbb{P}\left(\mathcal{E}_{1,i}\right)\leq \Delta_\gamma\max_{i\in\{1,\cdots,\Delta_\gamma\}}\mathbb{P}(\mathcal{E}_{1,i}).
\end{equation*}
By \eqref{eq: covering}, the above inequality implies \eqref{ineq: claim11_tl}, which concludes this proof.
\end{proof}

Based on Claim \ref{claim4_tl}, we aim to find a uniform upper bound for $\left\{\mathbb{P}\left(\mathcal{E}_{1,i}\right)\right\}_{i=1}^{\Delta_\gamma}$. To address this problem, we provide a temporal discretization with respect to the $k$ time segments with length $m$.

\noindent\textbf{Step 3: Discretization on time.} Specifically, we partition the time $[0,T]$ into segments with length $m$, and we construct one event per time segment such that the intersection of all these events form a superset of $\mathcal{E}_{1,i}$. We make the following claim.
\begin{claim}\label{claim5_tl}
    $\forall\, i\in\{1,\cdots,\Delta_\gamma\}, k\in\left\{0,\cdots,\left\lceil \frac{T}{m}\right\rceil - 1\right\}$, define
\begin{equation*}
    G_{i,k}:=\left\{h^\top_{i,km+L_{i,k}}(w_{km+L_{i,k}} -c_{i,km+L_{i,k}})\geq\frac{a_1\delta\xi}{4}\right\}\cap\mathcal{E}_2.
\end{equation*}
Then $\forall\, i\in\{1,\cdots,\Delta_\gamma\}$, we have
    \begin{equation}\label{eq: claim12_tl}
    \mathcal{E}_{1,i}\subseteq \bigcap_{k=0}^{\left\lceil \frac{T}{m}\right\rceil - 1}G_{i,k}.
    \end{equation}
\end{claim}

\noindent
\begin{proof}[Proof of Claim \ref{claim5_tl}]
    Recall that by the algorithm \eqref{equ: SME algo}, and Assumption \ref{ass: tbtl}, $\forall\, t\geq 0$,
\begin{equation*}
    w_t-\gamma z_t = x_{t+1}-\hat\theta z_t\in\mathbb{W}\subseteq\bigcap_{j=1}^{\mathcal{X}}H(c_j,h_j).
\end{equation*}
Therefore, $\forall (c,h)\in\{(c_j,h_j)\}_{j=1}^{\mathcal{X}}$, $h^\top(w_t - \gamma z_t)\geq g^\top c$. Thus, we have
\begin{equation*}
    h^\top_{i,km+L_{i,k}}(w_{km+L_{i,k}} -\gamma z_{km+L_{i,k}} - c_{i,km+L_{i,k}})\geq 0
\end{equation*}
It follows that, when $\mathcal{E}_{1,i}$ holds, by \eqref{eq:claim1_tl}, we have
\begin{align}\label{ineq: claim2_tl}
    &h^\top_{i,km+L_{i,k}}(w_{km+L_{i,k}}-c_{i,km+L_{i,k}})\nonumber\\
    \geq &h^\top_{i,km+L_{i,k}}(\gamma z_{km+L_{i,k}})\nonumber\\
    \geq &\frac{a_1\delta\xi}{4}.
\end{align}
Recall that the events $\{G_{i,k}\}_{k=0}^{\lceil\frac{T}{m}\rceil-1}$ are defined in in the following manner:
\begin{equation*}
    \left\{h^\top_{i,km+L_{i,k}}(w_{km+L_{i,k}} -c_{i,km+L_{i,k}})\geq\frac{a_1\delta\xi}{4}\right\}\cap\mathcal{E}_2.
\end{equation*}
Then, we can deduce that $\forall\, i\in\{1,\cdots,\Delta_\gamma\}$,
\begin{equation*}
    \mathcal{E}_{1,i}\subseteq \bigcap_{k=0}^{\left\lceil \frac{T}{m}\right\rceil - 1} G_{i,k}.
\end{equation*}
\end{proof}
Next, we aim to find an upper bound for the right hand side of \eqref{eq: claim12_tl}. We make the following claim.
\noindent
\begin{claim}\label{claim6_tl}
    For any $i\in\{1,\cdots, \Delta_\gamma\}$:
\begin{equation*}
    \mathbb{P}\left(\bigcap_{k=0}^{\left\lceil \frac{T}{m}\right\rceil - 1} G_{i,k}\right)\leq \left(1 - p_w\left(\frac{a_1\delta\xi}{4}\right)\right)^{\left\lceil\frac{T}{m}\right\rceil - 1}.
\end{equation*}
\end{claim}

\begin{proof}[Proof of Claim \ref{claim6_tl}]
    By Bayes' Theorem, $\forall\, i$:
\begin{equation*}
    \mathbb{P}\left(\bigcap_{k=0}^{\left\lceil \frac{T}{m}\right\rceil - 1} G_{i,k}\right) = \mathbb{P}\left(G_{i,0}\right)\displaystyle\prod_{k=1}^{\left\lceil \frac{T}{m}\right\rceil - 1}\mathbb{P}\left(G_{i,k}\middle| \bigcap_{\ell=0}^{k-1}G_{i,\ell}\right).
\end{equation*}

Recall that $\mathcal{E}_2\subseteq \left\{1\leq L_{i,k} \leq m,\ \forall\,i,k\right\}$. Hence, $G_{i,k}\subseteq \left\{h^\top_{i,km+L_{i,k}}(w_{km+L_{i,k}} -c_{i,km+L_{i,k}})\geq\frac{a_1\delta\xi}{4}\right\}\cap\left\{1\leq L_{i,k} \leq m\right\}$. For any $k\in\left\{2,3,\cdots,\left\lceil\frac{T}{m}\right\rceil - 1\right\}$, we have

\begin{align*}
    &\mathbb{P}\left(G_{i,k}\middle|\bigcap_{j=0}^{k-1}G_{i,j}\right)\\
    \leq&\mathbb{P}\left(h^\top_{i,km+L_{i,k}}(w_{km+L_{i,k}} -c_{i,km+L_{i,k}})\geq\frac{a_1\delta\xi}{4};\right.\\
    &\left. 1\leq L_{i,k}\leq m\middle| \bigcap_{j=0}^{k-1}G_{i,j}\right)\\
    =&\sum_{l=1}^m\!\mathbb{P}\!\left(\!h^\top_{i,km+L_{i,k}}(w_{km+L_{i,k}} -c_{i,km+L_{i,k}})\geq\frac{a_1\delta\xi}{4};\right.\\
    &\left.L_{i,k}=l \middle| \bigcap_{j=0}^{k-1}G_{i,j}\right)\\
    = &\sum_{l=1}^m\mathbb{P}\left(h^\top_{i,km+L_{i,k}}(w_{km+L_{i,k}} -c_{i,km+L_{i,k}})\geq\frac{a_1\delta\xi}{4}\middle|\right.\\
    &\left. L_{i,k}=l, \bigcap_{j=0}^{k-1}G_{i,j}\right) \times\mathbb{P}\left(L_{i,k}=l\middle| \bigcap_{j=0}^{k-1}G_{i,j}\right).
\end{align*}
The first equation is deduced by the law of total probability, and the second equation is deduced by Bayes' theorem. Meanwhile, notice that
\begin{align*}
    &\mathbb{P}\left(h^\top_{i,km+L_{i,k}}(w_{km+L_{i,k}} -c_{i,km+L_{i,k}})\geq\frac{a_1\delta\xi}{4}\right.\\
    &\left.\middle| L_{i,k}=l, \bigcap_{j=0}^{k-1}G_{i,j}\right)\\
    \overset{(f)}{=}&\int_{h_{0:km+l}}\mathbb{P}\left(h^\top_{i,km+L_{i,k}}(w_{km+L_{i,k}} -c_{i,km+L_{i,k}})\geq\frac{a_1\delta\xi}{4},\right.\\
    &\left.w_{0:km+l}=h_{0:km+l}\middle| L_{i,k}=l,\bigcap_{j=0}^{k-1}G_{i,j}\right)dh_{0:km+l}\\
    \overset{(g)}{=}&\int_{Q_{k,l}}\!\!\mathbb{P}\!\left(\!h^\top_{i,km+L_{i,k}}\!(w_{km+L_{i,k}}\! -\!c_{i,km+L_{i,k}})\!\geq\!\frac{a_1\delta\xi}{4}\!\middle|\! w_{0:km+l}\right.\\
    &\left.=h_{0:km+l}\right)\!\cdot\! \mathbb{P}\!\left(\!w_{0:km+l}=h_{0:km+l}\!\middle|\! L_{i,k}=l\!,\!\bigcap_{j=0}^{k-1}G_{i,j}\!\right)\\
    &dh_{0:km+l}\\
    \leq&\left(1-p_w\left(\frac{a_1\delta\xi}{4}\right)\right)\times\\
    &\int_{Q_{k,l}}\mathbb{P}\left(w_{0:km+l}=h_{0:km+l}\middle| L_{i,k}=l,\bigcap_{j=0}^{k-1}G_{i,j}\right)dh_{0:km+l}\\
    =&1-p_w\left(\frac{a_1\delta\xi}{4}\right),
\end{align*}
where (f) is deduced by the law of total probability, and (g) is deduced by Bayes' theorem. Specially, we denote a sequence of noise (or its realization) by
\begin{equation*}
    w_{0:km+l}:=\{w_0,w_1,\cdots, w_{km+l-1}\},
\end{equation*}
and the domain of integration by
\begin{align*}
    Q_{k,l}:=&\{h_{0:km+l}\text{ such that }w_{0:km+l}=h_{0:km+l}\text{ satisfies }\\
    &L_{i,k}=l\text{ and }\bigcap_{j=1}^{k-1}G_{i,j}\}.
\end{align*}
This follows that
\begin{align*}
    &\mathbb{P}\left(G_{i,k}\middle|\bigcap_{j=0}^{k-1}G_{i,j}\right)\\
    \leq&\left(1-p_w\left(\frac{a_1\delta\xi}{4}\right)\right)\sum_{l=1}^m\mathbb{P}\left(L_{i,k}=l\middle|\bigcap_{j=0}^{k-1}G_{i,j}\right)\\
    =&1-p_w\left(\frac{a_1\delta\xi}{4}\right).
\end{align*}
By a similar procedure, we can also deduce that
\begin{equation*}
    \mathbb{P}\left(G_{i,0}\right) \leq 1-p_w\left(\frac{a_1\delta\xi}{4}\right).
\end{equation*}
It follows that 
\begin{equation}\label{ineq: claim22_tl}
    \mathbb{P}\left(\bigcap_{k=0}^{\left\lceil \frac{T}{m}\right\rceil - 1} G_{i,k}\right)\leq \left(1 - p_w\left(\frac{a_1\delta\xi}{4}\right)\right)^{\left\lceil\frac{T}{m}\right\rceil - 1}.
\end{equation}
\end{proof}

\noindent
\textbf{Step 4: The final assembly.}
Combining \eqref{eq: events covered_tl}, \eqref{ineq: claim11_tl} and \eqref{ineq: claim22_tl}, we can prove Lemma \ref{lm3_tl}. Combining Lemma \ref{lm3_tl} with Lemma 4.1 in \cite{li2024setmembership}, we can prove Theorem \ref{thm2}.

\subsection{Proof of the strict positivity of $\xi$}\label{app: xi}
Recall that in the statement of Theorem \ref{thm2}, it is claimed that $\xi$ is strictly positive. Meanwhile, the proof of Theorem \ref{thm2} in Appendix \ref{app: thm2} also requires $\xi>0$. We present proofs of the strict positivity of $\xi$ under both circumstances: (i) there are finitely many SHS, and (ii) there are inifinitely many SHS.
\subsubsection{$\mathcal{H}$ is finite}
Recall that $\xi = \min_{\vv x\vv_2 = 1}\max_{h\in\mathcal{H}}h^\top x$. We show  $\forall\, x\in\mathbb{S}_2$, $\max_{h\in\mathcal{H}}h^\top x >0$ by contradiction. Suppose $\exists x\in\mathbb{S}_2$ such that $\forall\, h\in\mathcal{H}$, $h^\top x\leq 0$. It follows that $\forall w\in\W,\ n>0$, $w-nx\in\W$. This result contradicts the assumption that $\W$ is compact. Hence, $\xi$ must be strictly positive.

\subsubsection{$\mathcal{H}$ is infinite}
In this case, recall that $\xi = \min_{\vv x\vv_2 =1}\sup_{h\in\mathcal{H}}h^\top x$. We prove by contradiction. $\forall\epsilon>0$, suppose that $\exists \{x_n\}_{n=1}^\infty$ such that $\vv x\vv_2 = 1$ and $\forall h\in\mathcal{H}, n\geq 1$, $h^\top x_n \leq \frac{\epsilon}{n}$. Let $w$ be an interior point of $\W$. Then $\forall n\geq 1$, $w + n x_n \in\W$. This result contradicts the assumption that $\W$ is compact. Hence, $\xi$ must be strictly positive.

\subsection{Proof of Corollary \ref{cor: qw=O(epsilon)1} and Corollary \ref{cor: qw=O(epsilon)}}\label{app: cor}
First, we prove Corollary \ref{cor: qw=O(epsilon)1}. Recall that in this case we have $\xi = 1$. We first claim that
$    \text{Term 1} \leq \epsilon
$. To show this, notice that
    $m \geq O(n_z+\log T - \log\epsilon) =\frac{1}{a_3} \left[O((\log a_2)n_z + \frac{5}{2}\log n_z + O (\log T) - O(\log\epsilon)\right]$. It follows that $\exp(-a_3m) \leq \frac{a_2^{-n_z}n_z^{-5/2}\epsilon}{T}$. Therefore, $\text{Term 1} \leq \epsilon/m\leq \epsilon$. Next we let $\text{Term 2} = \epsilon$. Then $\epsilon = \tilde{O}\left((n_xn_z)^{5/2}\right)\tilde{a}_4^{n_xn_z}\left[1 - q_w\left(\frac{a_1\delta}{4}\right)\right]^{\left\lceil\frac{T}{m}\right\rceil - 1}$. Given $q_w(\epsilon) = O(\epsilon^p)$, we have:
\begin{align*}
    \delta^p &= O\left(\left(\frac{4}{a_1}\right)^p\right)\left\{1 - \left[\epsilon\tilde{O}\left((n_xn_z)^{5/2}\right)a_4^{n_xn_z}\right]^{\frac{1}{\left\lceil \frac{T}{m}\right\rceil - 1}}\right\}\\
    &\leq O\left(\left(\frac{4}{a_1}\right)^p\right)\left\{1 - \left[\epsilon\tilde{O}\left((n_xn_z)^{5/2}\right)a_4^{n_xn_z}\right]^{\frac{T}{m}}\right\}\\
    &\overset{(b)}{\leq} O\left(\left(\frac{4}{a_1}\right)^p\right)\cdot \frac{m}{T}\cdot\tilde{O}(n_xn_z)=\tilde{O}\left(\left(\frac{4}{a_1}\right)^p\cdot\frac{n_xn_z}{T}\right).
\end{align*}
Inequality (b) is deduced from the fact that $\forall\,x>0,\ x-1\geq\log x$. It follows that $\delta \leq \tilde{O}\left(\left(\frac{n_xn_z}{T}\right)^{\frac{1}{p}}\right)$. Then, we prove Corollary \ref{cor: qw=O(epsilon)1}.

Next, we show Corollary with the same reasoning, and by replacing $a_4$ with $a_5$, $\delta$ with $\xi\delta$, and $q_w(\epsilon)$ with $p_w(\epsilon)$.

\subsection{Proof of Examples \ref{eg1_1}, \ref{eg2_1}, and \ref{eg1}}\label{box}

The calculation of the corresponding boundary-visiting probability bounds below (i.e., $q_w(\epsilon)$ and $p_w(\epsilon)$) below requires some computations of general hyperspherical caps. We refer the reader to \cite{concise2011formulas}.

For the weighted $\ell_\infty$ ball $\W = \left\{w\in\mathbb{R}^{n_x}\,:\, \max_{1\leq i\leq n_x} \left\{\frac{|w_i|}{a_i}\right\}\leq 1\right\}$, and $w_t$ uniformly distributed on $\W$, the probability density of $w_t$ on $\W$ is $f_w(x) = \frac{1}{2^{n_x}\prod_{i=1}^{n_x}a_i}$.

\noindent\textbf{Example \ref{eg1_1}.}
In this case, we consider any vertex $w_v$ of $\W$, the probability that $w_t$ visits $\B_2^\epsilon(w_v)\cap\W$ is $\frac{1}{2^{n_x}}\cdot\frac{\pi^{\!\frac{n}{2}}\epsilon^{n_x}}{\Gamma\left(\frac{n}{2}+1\right)} \cdot \frac{1}{2^{n_x}\prod_{i=1}^{n_x}a_i}\sim O(\epsilon^{n_x})$.

\vspace{3pt}

\noindent\textbf{Example \ref{eg1}, weighted $\mathbf{\ell_\infty}$ ball.} To obtain the best theoretical guarantee for the convergence rate, we can choose one non-extreme point on each facet of $\W$ as $c$, and the unit normal vector of the corresponding facet at $c$ as $h$. The $\epsilon$-slice induced by $H(c,h)$ can be then written as $[-a_1,a_1]\times \cdots [a_i - \epsilon, a_i]\times [-a_{i+1},a_{i+1}]\times [-a_{n_x}, a_{n_x}]$ for some $i\in\{1,\cdots, n_x\}$ (or with $[-a_i, -a_i + \epsilon]$). The probability that $w_t$ visits this $\epsilon$-slice is $\frac{\epsilon\prod_{j\neq i}2a_j}{2^{n_x}\prod_{j=1}^{n_x}a_j} = \frac{\epsilon}{2 a_i}\sim O(\epsilon)$. The according projection constant is $\xi = \frac{1}{n_x}$. Therefore, by Corollary \ref{cor: qw=O(epsilon)}, the convergence rate is $\tilde{O}\left(\frac{n_x^{3/2}n_z}{T}\right)$.

For the weighted $\ell_1$ ball $\W = \left\{w\in\mathbb{R}^{n_x}\ :\ \sum_{i=1}^{n_x}\frac{|w_i|}{a_i} \leq 1\right\}$, and $w_t$ uniformly distributed on $\W$, the probability density of $w_t$ on $\W$ is $f_w(x) = \frac{n!}{2^n \prod_{i=1}^{n_x}a_i}$.

\noindent\textbf{Example \ref{eg2_1}.} In this case, the least intersection of a boundary $\epsilon$-ball with $\W$ is on one of the vertices of $\W$. Figure \ref{fig: sectors} demonstrates such an $\epsilon$-ball when $n_x=2$.
\begin{figure}[!htb]
    \centering
    \includegraphics[width=0.8\linewidth]{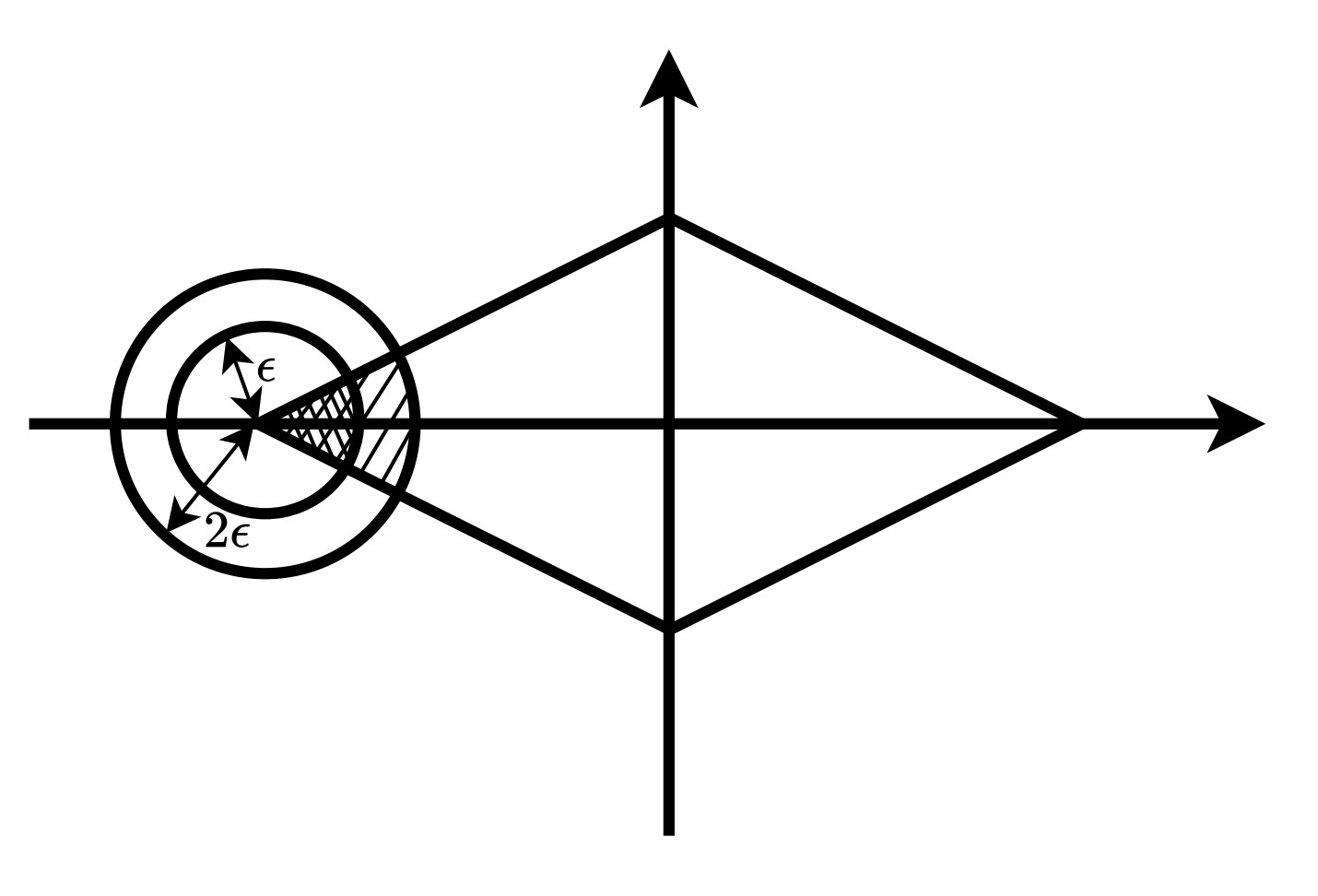}
    \caption{The greater shaded sector represents the intersection of the $2\epsilon$-ball with $\W$, while the smaller doubly shaded sector represents the intersection of the $\epsilon$-ball with $\W$.}
    \label{fig: sectors}
\end{figure}
Denote the doubly shaded cone (sector) by $S_\epsilon$. Notice that when we replace $\epsilon$ with $2\epsilon$, the intersection area increases to $2^{n_x}S_\epsilon$. It follows that $q_w(\epsilon) = O\left(\epsilon^{n_x}\right)$.

\noindent\textbf{Example \ref{eg1}, weighted $\mathbf{\ell_1}$ ball.} This is similar to the weighted $\ell_\infty$ ball case. The best choice of $(c,h)$ is: choose exactly on extreme point of each facet of $\W$ as $c$, and its according normal vector as $h$. Any $\epsilon$-slice can be viewed as the intersection of $\W$ and the slab generated by translating a supporting hyperplane by $\epsilon$. Due to the symmetric and parallel property of $\W$, every such $\epsilon$-slice carries a probability measure of $O(\epsilon)$.

\subsection{Proof of Example \ref{eg3_1} and Example \ref{eg3}}\label{2norm}
Recall that the $\ell_2$ ball is $\W = \left\{w\in\mathbb{R}^{n_x}\ :\ \vv w\vv_2 \leq 1\right\}$, and $w_t$ is uniformly distributed on $\W$. The probability density of $w_t$ on $\W$ is $f_w(x) = \frac{\Gamma\left(\frac{n_x}{2}+1\right)}{\pi^{\frac{n_x}{2}}}$.

\noindent\textbf{Example \ref{eg3_1}.} The volume of the intersection of $\W$ and any $\epsilon$-ball is
\begin{equation*}
    V_{\text{ball}}(\epsilon) = \frac{\pi^{n_x}}{\Gamma\left(\frac{n_x+1}{2}\right)}\left(\int_0^{T_1}\sin^{n_x}\theta d\theta + \epsilon^{n_x}\int_0^{T_2}\sin^{n_x}\theta d\theta\right),
\end{equation*}
where $T_1 = \arccos\left(1-\frac{\epsilon^2}{2}\right)$, $T_2 = \arccos\left(\frac{\epsilon}{2}\right)$. When $\epsilon$ is small, $V_{\text{ball}}(\epsilon)\sim O\left(\epsilon^{n_x}\right)$. It follows that $q_w(\epsilon) = O(\epsilon^{n_x})$.

\noindent\textbf{Example \ref{eg3}.} The volume of the $\epsilon$-slice of $\W$ at any boundary point of $\W$ is
\begin{equation*}
    V_{\text{slice}} = \frac{\pi^{n_x}}{\Gamma\left(\frac{n_x+1}{2}\right)}\int_0^{T_3}\sin^{n_x}\theta d\theta,
\end{equation*}
where $T_3 = \arccos(1-\epsilon)$. When $\epsilon$ is small enough, we have
$V_{\text{slice}}(\epsilon)\sim O\left(\epsilon^{\frac{n_x+1}{2}}\right)$. It follows that $p_w(\epsilon) = O\left(\epsilon^{\frac{n_x+1}{2}}\right)$.

The above calculations of the volumes of hypersphereical caps follow from the formulas provided in \cite{concise2011formulas}.

\subsection{More discussions on Assumption \ref{ass: tbtl}}\label{app: epsilon}
In the discussion on Assumption \ref{ass: tbtl}, it is claimed that for the same $\epsilon$ value and the same boundary point of $\W$, the $\epsilon$-ball is a subset of any of its $\epsilon$-slice. We hereby present a proof for this statement.

Let $\W\in\mathbb{R}^n$ be a convex and compact set, and let $w^0$ be a boundary point of $\W$. Let $\{x\in\mathbb{R}^n\ :\ h^\top x\geq h^\top w^0\}$ be a supporting half-space of $\W$ at $w^0$. For any fixed $\epsilon >0$, consider the $\epsilon$-ball centered at $w^0$:
\begin{equation*}
    \B_\epsilon^2(w^0) = \{x\in\mathbb{R}^n\ :\ \vv x - w^0\vv_2 \leq \epsilon\},
\end{equation*}
and the $\epsilon$-slice induced by $(w^0, h)$:
\begin{equation*}
    S_{\W}^\epsilon(w^0, h) := \{x\in\mathbb{R}^n\ : h^\top w^0 + \epsilon \geq h^\top x \geq h^\top w^0\}.
\end{equation*}
Then, we show that $\B_\epsilon^2(w^0)\cap\W \subseteq S_{\W}^\epsilon(w^0, h)\cap\W$. To see this, recall that $\forall x\in \W$, we have $h^\top x \geq h^\top x^0$ since $S_{\W}^\epsilon(w^0, h)$ is a SHS of $\W$. Therefore, it suffices to show that $\B_\epsilon^2(w^0)\subseteq\{x\in\mathbb{R}^n\ :\ h^\top w^0 + \epsilon \geq h^\top x\}$. $\forall\, x\in \B_\epsilon^2(w^0)$, we have $\vv x - w^0\vv_2 \leq \epsilon$. By the Cauchy-Schwartz inequality, we have
\begin{equation*}
    h^\top(x - w^0) \leq \vv h\vv_2\cdot\vv x-w^0\vv_2\leq 1\cdot\epsilon= \epsilon.
\end{equation*}

\end{document}